\documentclass[journal,twoside,web]{ieeecolor}
\usepackage{generic}
\usepackage{cite}
\usepackage{amsmath,amssymb,amsfonts}
\usepackage{algorithmic}
\usepackage{graphicx}
\usepackage{algorithm,algorithmic}
\usepackage{hyperref}
\usepackage{textcomp}
\usepackage{xcolor}

\usepackage{algorithm}
\usepackage{algorithmic} 
\usepackage{float}
\usepackage{float}
\usepackage{subfig}
\usepackage{caption}
\usepackage{tabularx}
\usepackage[flushleft]{threeparttable} 
\let\labelindent\relax
\usepackage{enumitem}
\usepackage{makecell} 
\usepackage{booktabs}

\usepackage{url}
\usepackage{color,hyperref}
\definecolor{darkblue}{rgb}{0,0,0.8}

\newtheorem{theorem}{Theorem}
\newtheorem{lemma}{Lemma}
\newtheorem{assumption}{Assumption}

\newtheorem{remark}{Remark}
\newtheorem{corollary}{Corollary}

\newcommand{\order}[1]{\mathcal{O}\left(#1\right)}
\newcommand{\prt}[1]{\left(#1\right)}
\newcommand{\brk}[1]{\left[#1\right]}
\newcommand{\crk}[1]{\left\{#1\right\}}

\newcommand{\inpro}[1]{\left\langle #1 \right\rangle}
\newcommand{\condE}[2]{\E\brk{#1\middle|#2}}
\newcommand{\norm}[1]{\left\Vert #1 \right\Vert}
\newcommand{\bs}[1]{\boldsymbol{#1}}

\newcommand{\normi}[1]{\Vert #1 \Vert}
\newcommand{\crki}[1]{\{#1\}}

\newcommand{\brki}[1]{[#1]}
\newcommand{\orderi}[1]{\mathcal{O}(#1)}
\newcommand{\condEi}[2]{\E\brki{#1|#2}}

\newcommand{\R}{\mathbb{R}}
\newcommand{\E}{\mathbb{E}}

\newcommand{\sumn}{\sum_{i=1}^n}

\newcommand{\cF}{\mathcal{F}}
\newcommand{\uf}{f^*}

\newcommand{\cL}{\mathcal{L}}

\newcommand{\xitl}{x_{i,t}^\ell}
\newcommand{\cG}{\mathcal{G}}
\newcommand{\cC}{\mathcal{C}}
\newcommand{\cO}{\mathcal O}

\newcommand{\cE}{\mathcal{E}}
\newcommand{\cB}{\mathcal{B}}
\newcommand{\teta}{\tilde{\eta}}
\newcommand{\sigfn}{\sigma_f^*}

\definecolor{cuhkpl}{RGB}{152,24,147}
\def\kh#1{\color{cuhkpl}{#1}}

\def\BibTeX{{\rm B\kern-.05em{\sc i\kern-.025em b}\kern-.08em
    T\kern-.1667em\lower.7ex\hbox{E}\kern-.125emX}}
\begin{document}
\title{Distributed Stochastic Optimization under a General Variance Condition}
\author{Kun Huang, Xiao Li, \IEEEmembership{Member, IEEE}, and Shi Pu, \IEEEmembership{Senior Member, IEEE}
\thanks{X. Li was partially supported by the National Natural Science Foundation of China (NSFC) under Grant No. 12201534 and by the Shenzhen Science and Technology Program under Grant No. RCBS20210609103708017. S. Pu was partially supported by Guangdong Talent Program under Grant 2021QN02X216, by the National Natural Science Foundation of China under Grant 62003287, and by Shenzhen Science and Technology Program under Grant RCYX202106091032290.}
\thanks{K. Huang, X. Li and S. Pu are with the School of Data Science, The Chinese University of Hong Kong,
Shenzhen (CUHK-Shenzhen), Shenzhen 518172, China (e-mail: kunhuang@link.cuhk.edu.cn; lixiao@cuhk.edu.cn; pushi@cuhk.edu.cn).}
}

\maketitle

\begin{abstract}
    Distributed stochastic optimization has drawn great attention recently due to its effectiveness in solving large-scale machine learning problems.  Though numerous algorithms have been proposed and successfully applied to general practical problems, their theoretical guarantees mainly rely on certain boundedness conditions on the stochastic gradients, varying from uniform boundedness to the relaxed growth condition. In addition, how to characterize the data heterogeneity among the agents and its impacts on the algorithmic performance remains challenging. In light of such motivations, we revisit the classical Federated Averaging (FedAvg) algorithm \cite{mcmahan2017communication} as well as the more recent SCAFFOLD method \cite{karimireddy2020scaffold} for solving the distributed stochastic optimization problem and establish the convergence results under only a mild variance condition on the stochastic gradients for smooth nonconvex objective functions. Almost sure convergence to a stationary point is also established under the condition. Moreover, we discuss a more informative measurement for data heterogeneity as well as its implications. 
\end{abstract}


\section{Introduction}
\label{sec:introduction}
    \IEEEPARstart{W}{e} consider solving the following optimization problem by a group of agents $[n]:=\crki{1,2,\cdots, n}$: 
	\begin{equation}
		\label{eq:P}
		\min_{x\in\R^p} f(x):=  \frac{1}{n} \sumn f_i(x),
	\end{equation}
	where each $f_i:\R^p\rightarrow \R$ is the local cost function known by agent $i$ only. Solving problem \eqref{eq:P} with or without a central coordinator has attracted great interest recently due to its wide applications in multi-agent target seeking{\cite{pu2018swarming}}, distributed estimation{\cite{plata2015distributed}}, machine learning{\cite{chowdhery2022palm,qin2016distributed}}, etc. In particular, we consider the distributed learning problem, where each $f_i$ is an expected risk function $f_i(x) = \E_{\xi} F_i(x;\xi)$ with $\xi$ representing the data, or an empirical risk function $f_i(x) ={\frac{1}{m_i}} \sum_{j=1}^{m_i} f_{i,j}(x)$. Under such formulations, querying noisy or stochastic gradients in the form of $g_i(x;\xi_i)$ is often more realistic compared to using full gradients, and thus distributed stochastic gradient methods have become the workhorse for solving problem \eqref{eq:P}.

	The study of distributed stochastic gradient methods typically relies on some boundedness conditions, i.e., the bounds on $\E_{\xi_i}[\normi{g_i(x;\xi_i) - \nabla f_i(x)}^2]$ and $\frac{1}{n} \sumn[\normi{\nabla f_i(x) - \nabla f(x)}^2]$ respectively. The first term  represents the variance of the stochastic gradients for each agent, and the second one measures the dissimilarity among the individual functions.
    These bounds are the key factors in characterizing the performance of an algorithm. Hence, it is necessary to relax both conditions simultaneously to have fewer restrictions on a distributed algorithm. In the literature, we often see the following two types of conditions; see, e.g., \cite{lian2017can,pu2021sharp,chen2022optimal}: 
	\begin{subequations}
		\label{eq:g_sep}
		\begin{align}
			&\E_{\xi_i}\brk{\norm{g_i(x;\xi_i) - \nabla f_i(x)}^2}\leq \sigma^2 + \eta^2\norm{\nabla f_i(x)}^2,\nonumber\\
            &\quad \forall i\in [n]\text{ and for some }\sigma, \eta\geq 0,\label{eq:intro_bv1}\\
			&\frac{1}{n} \sumn\brk{\norm{\nabla f_i(x) - \nabla f(x)}^2} \leq \zeta^2 + \psi^2\norm{\nabla f(x)}^2,\nonumber\\
            &\quad \text{ for some }\zeta, \psi\geq 0.\label{eq:intro_bgd1}
		\end{align}
	\end{subequations}
	Condition \eqref{eq:intro_bv1} is the so-called relaxed growth condition \cite{bottou2018optimization} on the stochastic gradient $g_i(x;\xi_i)$ of $f_i(x)$ and is usually hard to verify in practice even with both $\sigma, \eta> 0$. In fact, there is a simple finite-sum minimization problem that does not satisfy such a condition; see Proposition 1 in \cite{khaled2020better}. The second condition is the notorious $(\zeta^2,\psi^2)-$bounded gradient dissimilarity (BGD) assumption specific to the distributed optimization problem. It is often used to measure the data heterogeneity among the local datasets of the agents where larger $\zeta^2$ and $\psi^2$ indicate larger degree of non-i.i.d. 

    However, when using condition \eqref{eq:intro_bgd1} to measure the data heterogeneity, the constants $\zeta^2$ and $\psi^2$ are less informative if $\zeta,\psi>0$ and the summation $(\zeta^2 + \psi^2)$ can be invariant to different degrees of data heterogeneity, as elaborated in our numerical example in Subsection \ref{subsec:sim2}.
    Using \eqref{eq:intro_bgd1} directly as an assumption to derive convergence properties for distributed learning algorithms ignores how heterogeneous the data is as the relevant quantities are hidden in $\zeta$ and $\psi$, and hence it is meaningful to establish this condition for certain problem class and characterize the important constants $\zeta$ and $\psi$ more specifically. 
    It is worth noting that even though some previous works, e.g., \cite{huang2022improving,pu2021sharp,karimireddy2020scaffold,Li2020On,wang2022unreasonable,huang2022tackling} do not assume condition \eqref{eq:intro_bgd1}, the restriction is still significant due to condition \eqref{eq:intro_bv1}. These {observations} motivate our main contribution stated in the next subsection.

    \subsection{Main Contribution}
    The main contribution of this work is three-fold.
    Firstly, we discuss a more informative and practical measurement for data heterogeneity in distributed optimization for smooth nonconvex objective functions. Such a measurement is inspired by earlier works on distributed convex optimization or random-reshuffling methods, e.g., \cite{Li2020On,malinovsky2022server,huang2021distributed,das2022differentially}. The measurement only relies on the smoothness and lower boundedness of the objective functions. We also construct an example showing that such a measurement is not a trivial upper bound and clearly characterize its implications for the algorithmic performance.

    Secondly, for smooth nonconvex objective functions, we show both Federated Averaging (FedAvg) \cite{mcmahan2017communication} {and SCAFFOLD \cite{karimireddy2020scaffold}} maintain the same iteration complexity under a more general condition on the stochastic gradients and the new measurement for solving the distributed optimization problem \eqref{eq:P}. 
    Specifically, we consider the ABC condition \cite{khaled2020better,li2022unified,lei2019stochastic} in place of \eqref{eq:intro_bv1}. This relaxes the previous restrictions \eqref{eq:intro_bv1} and \eqref{eq:intro_bgd1} in this work. {Moreover, the impact of data heterogeneity is clearly demonstrated by examining and comparing the complexity results of FedAvg and SCAFFOLD. Interestingly, our analysis reveals that SCAFFOLD can also be affected by the data heterogeneity caused by the general variance condition on the stochastic gradients. We further support the theoretical results with numerical experiments.}

    Thirdly, we derive the almost sure convergence result for every single run of FedAvg under the decreasing stepsize policy for minimizing smooth nonconvex objective functions under the same general condition. Such a result complements the complexity result and provides a partial justification for the last iterate output by FedAvg.



\subsection{Related Works}

    We restrict our discussions to the first-order distributed optimization methods for solving problem \eqref{eq:P}. In general, the works for solving \eqref{eq:P} in a distributed manner can be classified into two categories based on whether a central server or coordinator is available. The algorithms that do not rely on a central coordinator are often referred to as decentralized methods, initiated from \cite{nedic2009distributed}. 
    Recently, many decentralized stochastic gradient methods, e.g., \cite{lian2017can,pu2021sharp,pu2021distributed,huang2022improving,yuan2020influence,alghunaim2021unified,tang2018d,li2019decentralized,xin2021improved,chen2015learning,chen2015learning2,koloskova2021improved} have been shown to achieve comparable performance to their centralized counterparts where centralized information aggregation is possible.
    Among these schemes, gradient tracking based algorithms \cite{pu2021distributed,xin2021improved,koloskova2021improved} and primal-dual like methods \cite{huang2022improving,yuan2020influence,tang2018d} can relieve from the data heterogeneity assumption \eqref{eq:intro_bgd1}, but they still require condition \eqref{eq:intro_bv1}. 

    When a central coordinator is available, Federated Averaging (FedAvg) \cite{mcmahan2017communication} is the most popular distributed learning method due to its simplicity and efficiency.  
    Nevertheless, condition \eqref{eq:intro_bgd1} is commonly assumed even for the simplest version of FedAvg, i.e., full user participation, equal local steps, and equal weighted objective functions; see, e.g., \cite{woodworth2020mini,yu2019linear}. For more involved settings, the options of partial participation and different local steps in FedAvg can introduce more degrees of data heterogeneity, and thus many variants 
    have been proposed to address the issue. {For instance, the FedProx method introduced in \cite{li2020federatedopt} enhances the practical performance compared to FedAvg under data heterogeneous settings but lacks satisfactory theoretical guarantee for smooth nonconvex objectives with noisy gradients. The SCAFFOLD method proposed in \cite{karimireddy2020scaffold} enhances both practical and theoretical performance of FedAvg when encountering data heterogeneity. The work in \cite{gorbunov2021local} proposed a unified framework for both methods under (strongly) convex setting. Such a framework also inspires other variants for Federated Learning. 
    For in-depth discussions, we refer to recent surveys such as \cite{grudzien2023can} and \cite{kairouz2021advances}. Nevertheless, condition {\eqref{eq:intro_bv1}} is still assumed in these cases.}


    There are some recent works aiming at relaxing condition \eqref{eq:intro_bgd1} for distributed stochastic gradient methods. For strongly convex objective functions, the works in \cite{Li2020On,wang2022unreasonable,pu2021sharp,huang2022tackling} have relaxed condition \eqref{eq:intro_bgd1} by considering other quantities related to the optimal solution to problem \eqref{eq:P} and the solution to problem $\min_x f_i(x)$. However, assumption \eqref{eq:intro_bv1} is still necessary in \cite{wang2022unreasonable,pu2021sharp,huang2022tackling}, and a more restricted uniformly bounded gradient is considered in \cite{Li2020On}. To relax assumption \eqref{eq:intro_bv1}, the so-called ABC condition is currently the most general alternative in distributed settings; see \cite{li2022unified,khaled2020better,lei2019stochastic}, and more detailed discussions about this condition can be found in \cite{khaled2020better,lei2019stochastic}. It is worth noting that there exists a more general variance assumption for traditional stochastic gradient methods \cite{patel2022global,patel2022stopping}, but it remains unclear how to obtain complexity results under the assumption.
    Another line of work avoids condition \eqref{eq:intro_bv1} by specifying a different sampling strategy, i.e., distributed random reshuffling-type methods \cite{huang2021distributed,malinovsky2022server}. 

    Table \ref{tab:fedavg} presents a detailed comparison between this work with state-of-the-art results in the literature.


    \begin{table*}[htbp]
        \setlength{\tabcolsep}{4pt}
        \setlength\extrarowheight{10pt}
        \centering
        \begin{threeparttable}
        \caption{A comparison of the current results for FedAvg/Local SGD type methods. All the methods assume each $f_i$ is smooth and $f$ is bounded from below. The column ``$g$" corresponds to the bound on the variance of the stochastic gradients, i.e., condition \eqref{eq:intro_bv1}.  The ``Non-IID" column states the measure for data heterogeneity among the agents. 
        The last column states whether there is a result showing $\lim_{t\rightarrow\infty}\norm{\nabla f(x_t)}=0$ almost surely. }
         \label{tab:fedavg}
    \begin{tabular}{@{}ccccccc@{}}
        \toprule
        Paper                    & $g$                  & {\makecell[c]{Additional\\ Assumption}}      & Non-IID                                               & $f$        & Rate                                                                                                                                                    & a.s. \\ \midrule
        \cite{wang2022unreasonable} & $(\sigma^2,0)$       & /                           & $\rho$\tnote{(a)}                                     & $f_i$ SCVX & $\tilde{\cO}\prt{\frac{{  \sigma^2}}{nT} + \frac{{  \sigma^2}}{T^2} + \rho^2}$                                                                                                & N    \\
        \cite{Li2020On}             & $G^2$\tnote{(b)}     & /                           & $f^* - \frac{1}{n}\sumn f_i^*$\tnote{(c)}             & $f_i$ SCVX & $\order{\frac{f^* - \frac{1}{n}\sumn f_i^* {  + G^2 + \sigma^2}}{T}}$                                                                                                        & N    \\
        \cite{khaled2020tighter}    & $(\sigma^2,0)$       & /                           & $\frac{1}{n}\sumn\norm{\nabla f_i(x^*)}^2$\tnote{(d)} & CVX        & $\cO\biggl(\frac{{  \sigma^2} + \frac{1}{n}\sumn\norm{\nabla f_i(x^*)}^2}{\sqrt{nT}}+ \frac{{  n(\frac{1}{n}\sumn\norm{\nabla f_i(x^*)}^2 + \sigma^2)}}{T}\biggr)$ & N    \\
        \cite{woodworth2020mini}    & $(\sigma^2,0)$       & $(\zeta^2, 0)$-BGD          & $(\zeta^2, 0)$-BGD                                    & CVX        & $\order{\frac{{  \sigma^2 }}{\sqrt{nT}} + \frac{\zeta^{\frac{2}{3}} + {  \sigma^{\frac{2}{3}}}}{T^{\frac{2}{3}}} } $                                                                          & N    \\
        \cite{yu2019linear}         & $(\sigma^2,0)$       & $(\zeta^2, 0)$-BGD          & $(\zeta^2, 0)$-BGD                                    & NCVX       & $\order{\frac{{  \sigma^2 }}{\sqrt{nT}} + \frac{({  \sigma^2} + \zeta^2)n}{T}}$                                                                                                & N    \\
        \cite{yang2021achieving}    & $(\sigma^2,0)$       & $(\zeta^2, 0)$-BGD          & $(\zeta^2, 0)$-BGD                                    & NCVX       & $\order{\frac{{  \sigma^2}}{\sqrt{nT}} + \frac{{ n(\zeta^2+\sigma^2)}}{T} }$                                                                                                      & N    \\
        \cite{wang2020tackling}     & $(\sigma^2,0)$       & $(\zeta^2, \psi^2 + 1)$-BGD & $(\zeta^2, \psi^2 + 1)$-BGD                           & NCVX       & $\order{\frac{{  \sigma^2}}{\sqrt{nT}} + \frac{{  n(\zeta^2+\sigma^2)}}{T}}$                                                                                                  & N    \\
        \cite{qin2022role}          & $(\sigma^2,0)$       & $(\zeta^2, \psi^2 + 1)$-BGD & $(\zeta^2, \psi^2 + 1)$-BGD                           & NCVX       & $\order{\frac{{  \sigma^2}+\zeta^2}{\sqrt{nQT}}}$                                                                                                                   & N    \\
        \cite{chen2022optimal}      & $(\sigma^2, \eta^2)$ & $(\zeta^2, 0)$-BGD          & $(\zeta^2, 0)$-BGD                                    & NCVX       & $\order{\frac{{  \eta^2 + \sigma^2}}{\sqrt{nT}} + \zeta^2}$                                                                                                                 & N    \\ \midrule
        \textbf{\makecell[c]{FedAvg\\ (This work)}}                       & \textbf{ABC}                  & $\bs{f_i}$ \textbf{lower bounded}\tnote{(e)}                           & $\boldsymbol{{ \sigfn}}$\tnote{(f)}               & \textbf{NCVX}       & $\boldsymbol{\order{\frac{{  D + {C\sigfn}} }{\sqrt{n{Q}T}} + {\frac{(1+C)\sigfn}{T^{2/3}}} + {\frac{nD}{T}} }}$\tnote{(g)}                                                                               & \textbf{Y}{\tnote{(h)}}    \\ 
        \textbf{{\makecell[c]{SCAFFOLD\\ (This work)}}}                       & \textbf{{ABC}}                  & {$\bs{f_i}$ \textbf{lower bounded}}                           & $\boldsymbol{{\sigfn}}$               & {\textbf{NCVX}}       & $\boldsymbol{{\order{\frac{{  D +  C\sigfn} }{\sqrt {nQT}} + \frac{C\sigfn}{(nQ^2)^{1/6}T^{2/3}} + \frac{D + C\sigfn}{T}}}}$                                                                               & {\textbf{N}}    \\ 
        \bottomrule
        \end{tabular}
        \begin{tablenotes}
            \item[(a)] Average drift at the optimum; see \eqref{eq:rho} for details.
            \item[(b)] The constant $G$ is defined such that $\E\brki{\norm{g_i(x;\xi)}^2}\leq G^2$.
            \item[(c)] We define $f^* := f(x^*)$ and $f_i^* := f_i(x_i^*)$, where $x^*$ is the optimal solution to $\min_x f(x)$ and $x_i^*$ is the optimal solution to $\min_x f_i(x)$.
            \item[(d)] 
            We choose $x^*$ as a fixed minimizer of $f$.
            \item[(e)] This is a mild assumption satisfied for many practical problems.
            \item[(f)] {We denote $\sigfn:= \uf- \sumn\uf_i/n$.} The quantities $f^*$ and $f_i^*$ are defined as $f^*:=\inf_x f(x)$ and $f_i^*:= \inf_x f_i(x)$. 
            \item[(g)] The constant $D$ is defined in the ABC assumption \ref{as:abc}.
            \item[(h)] {The result requires the stepsize sequence $\{\alpha_t\}$ to satisfy $\sum_{t=0}^{\infty}\alpha_t = \infty, \ \sum_{t=0}^\infty \alpha_t^2 <\infty.$}
          \end{tablenotes}
        \end{threeparttable}
    \end{table*}

    \section{Relaxed Conditions}
    
    In this section, we introduce the standing assumptions considered in this paper.
    
    The following assumption is the so-called ABC condition\footnote{The original ABC condition includes the term $B\norm{\nabla f_i(x)}^2$, which can be absorbed by the term $(f_i(x) - \uf_i)$ given smoothness. Hence, it is sufficient to consider the case in \eqref{eq:abc}.}, which is currently the most general bound on the variance of the stochastic gradients for each agent. In particular, it covers the counterexample given in Proposition 1 of \cite{khaled2020better}. 
            \begin{assumption}
                \label{as:abc}
                {Assume the generated stochastic process $\crki{x_{i,k}^\ell, i\in[n], \ell\in \crk{1,2,\cdots, Q-1}}$, where $Q$ denotes the number of local updates, is adapted to the filtration $\crki{\cF_k^\ell}$, and each stochastic gradient $g_i(x_{i,k}^{\ell};\xi_{i,k}^{\ell})$ is $\cF_{k + 1}^{\ell}-$measurable.} Each agent has access to an unbiased stochastic gradient $g_i(x;\xi_i^\ell)$ of $\nabla f_i(x)$, i.e., $\condEi{g_i(x;\xi_{i,k}^\ell)}{\cF_k^\ell} = \nabla f_i(x)$ almost surely, and there exists $C,D\geq 0$ such that for any $k\in\mathbb{N}$, $i\in[n]$, and $\ell\in\crki{0,1,\cdots, Q-1}$,
                \begin{align}
                    \label{eq:abc}
                    \condE{\norm{g_{i}(x;\xi_{i,k}^\ell) - \nabla f_i(x)}^2}{\cF_k^\ell}&\leq C\brk{f_i(x) - \uf_i} + D,\nonumber\\
                    &\quad \text{ almost surely},
                \end{align} 
                {where $f_i^* := \inf_x f_i(x)$.}
                In addition, the stochastic gradients are independent across different agents at each $k\geq 0$. 
            \end{assumption}
            
      The independence among  agents at each $k\geq 0$ is natural and commonly assumed in the distributed optimization literature; see, e.g., \cite{pu2021sharp,huang2022improving,pu2021distributed}. 
    
            \begin{remark}
                The filtration $\crki{\cF_k^\ell}$ has the relation 
                \begin{align*}
                    \cF_0^0 \subset \cF_0^1\subset\cdots \subset\cF_0^{Q-1}\subset\cF_1^0\subset\cF_1^1\subset\cdots\subset\cF_k^{\ell} \subset\cdots
                \end{align*}
                We denote $\cF_k:=\cF_k^0, \forall k$ for simplicity to emphasize the server's iterates.
            \end{remark}
            
        {Our second assumption} is the common smoothness and lower boundedness condition for studying nonconvex objective functions. 
            \begin{assumption}
                \label{as:smooth}
                Each $f_i(x):\R^p\rightarrow\R$ is $L$-smooth, i.e., 
                \begin{align*}
                    \norm{\nabla f_i(x) - \nabla f_i(x')}\leq L\norm{x - x'}, \ \forall x,x'\in\R^p.
                \end{align*}
                In addition, each $f_i(x)$ is bounded from below, i.e., $f_i(x)\geq \uf_i >-\infty$, $\forall x\in\R^p$.
            \end{assumption}
              \begin{remark}
                Condition \eqref{eq:intro_bv1} together with Assumption \ref{as:smooth} implies \eqref{eq:abc}. {This is due to the descent lemma \cite{nesterov2003introductory}:
                \begin{equation}
                    \label{eq:descent_lemma}
                    \begin{aligned}
                        f_i(y) \leq f_i(x) + \inpro{\nabla f_i(x), y-x} + \frac{L}{2}\norm{y-x}^2,\ \forall x,y\in\R^p.
                    \end{aligned}
                \end{equation}
                Taking $y = x - \nabla f_i(x)/L$ and noting that $f_i(y)\geq f_i^*$ lead to 
                \begin{align}
                    \label{eq:nf}
                    \norm{\nabla f_i(x)}^2\leq 2L\prt{f_i(x) - f_i^*}.
                \end{align}
                }
            \end{remark}
            \begin{remark}
                \label{rem:smooth}
                Assumption \ref{as:smooth} indicates that $f$ has a lower bound $\frac{1}{n}\sumn \uf_i$. In addition, we have $\uf:= \inf_x f(x)\geq \frac{1}{n}\sumn \uf_i$ and $\normi{\nabla f_i(x)}^2$ is upper bounded, i.e.,{\eqref{eq:nf}.}
                Such a relationship naturally leads to an upper bound for $\frac{1}{n} \sumn[\normi{\nabla f_i(x) - \nabla f(x)}^2]$, and hence one can replace the RHS in condition \eqref{eq:intro_bgd1} with a more quantitative measure. Specifically, we have the following result.
            \end{remark}
       \begin{lemma}
            \label{lem:ngrad}
            Let Assumption \ref{as:smooth} hold. We have 
            \begin{multline}
                    \frac{1}{n}\sumn\norm{\nabla f_i(x) - \nabla f(x)}^2
                    \leq 2L\prt{f(x) - \uf} + 2L{\sigfn},
            \end{multline}
            {where $\sigfn=\uf - \frac{1}{n}\sumn \uf_i$.}
        \end{lemma}
     \begin{proof}
           See Appendix \ref{app:lem_ngrad}.
        \end{proof}


    Note that Assumptions
    \ref{as:abc} and \ref{as:smooth} can be verified in practice for many concrete problems. For example,the following lemma introduces such a condition that was first considered in \cite{lei2019stochastic}.
            \begin{lemma}
                \label{lem:h_con}
                Suppose each $f_i(x)$ has the finite sum structure, i.e., $f_i(x):= \frac{1}{m_i}\sum_{j=1}^{m_i} f_{i,j}(x)$. We define $F_i(x;\xi):= \frac{1}{m_i}\sum_{j=1}^{m_i}[\xi]_jf_{i,j}(x)$ {as in \cite{khaled2020better,richtarik2020stochastic,gower2019sgd,singh2023mini}}, where $\xi\in\R^{m_i}$ is a random vector corresponding to some sampling strategies. Assume  $\E[\xi]_j = 1$ for any $[\xi]_j$, and {the gradient of $F_i(x;\xi)$
                is Lipschitz continuous}, that is, for any $x,x'\in\R^p$ and $\xi$,
                \begin{align*}
                    \norm{\nabla F_i(x;\xi) - \nabla F_i(x';\xi)}_2\leq L(\xi)\norm{x-x'}_2,
                \end{align*}
                where $\nabla F_i(x;\xi) :=  \frac{1}{m_i}\sum_{j=1}^{m_i}[\xi]_j\nabla f_{i,j}(x)$.
                Then, $f_i(x)= \E_\xi F_i(x;\xi)$ is $L$-{smooth}, where $L=\E_\xi L(\xi)$. In addition, Assumption \ref{as:abc} is satisfied if:
                \begin{enumerate}
                    \item  the random variable $[\xi]_j$ is obtained by independent sampling with replacement, then $C= \max_i\frac{L}{C_i}$ and $D=2C(\uf - \frac{1}{n}\sumn\uf_i)$, where $C_i>0$ is some constant,
                    \item or $[\xi]_j$ is obtained by independent sampling without replacement, then $C=\max H_i L$ and $D=2C(\uf - \frac{1}{n}\sumn\uf_i)$, where $H_i>0$ is some constant. 
                \end{enumerate}
            \end{lemma}
            
            \begin{remark}
                It is worth noting that the constants $C_i$ and $H_i$ are corresponding to the sampling strategies and are independent of $f_i^*$ and $f^*$. The condition that the gradient of $F_i(x;\xi)=\frac{1}{m_i}\sum_{j=1}^{m_i}[\xi]_jf_{i,j}(x)$ being Lipschitz continuous is mild here. For example, assuming each $f_{i,j}$ is a quadratic loss or logistic loss satisfies such a condition.
            \end{remark}
            
            \begin{proof}
                See Appendix \ref{app:lem_h_con}.
            \end{proof}

            \section{Measuring Data Heterogeneity}

            In earlier works for studying distributed nonconvex optimization, the common condition \eqref{eq:intro_bgd1} not only provides a bound for the gradient dissimilarity, but also suggests a measure for the data heterogeneity among the agents. However, the meanings of the parameters $\zeta^2$ and $\psi^2$ in condition \eqref{eq:intro_bgd1} are rather vague and often conservative.
                 Recall from Lemma \ref{lem:ngrad} that
              \begin{multline}
                        \label{eq:bgd1}
                            \frac{1}{n}\sumn\norm{\nabla f_i(x) - \nabla f(x)}^2
                            \leq 2L\prt{f(x) - \uf} + 2L{\sigfn}.
                    \end{multline}
             We consider the following quantity to measure data heterogeneity:
            \begin{equation}
            \label{eq:NonIID_measure}
                \sigfn=\uf - \frac{1}{n}\sumn \uf_i.
            \end{equation}
                
             {Such a quantity was introduced earlier under smooth strongly convex objective functions (see \cite{Li2020On}) and for random reshuffling methods (see \cite{malinovsky2022server}). This work generalizes the notion to the nonconvex setting.} 
             We will show that quantity \eqref{eq:NonIID_measure} provides a more informative measurement for data heterogeneity among the agents.

                \begin{remark}
                    Compared to the parameter $\zeta^2$ in condition \eqref{eq:intro_bgd1}, the quantity $(\uf - \frac{1}{n}\sumn \uf_i)$ can be estimated more easily. On one hand, $\uf_i$ equals $0$ for problems such as overparameterized neural networks \cite{song2023fedavg}. On the other hand, we can estimate $\uf$ using $f(x_T)$, where $x_T$ is the last iterate after running an optimization algorithm for $T$ iterations.
                \end{remark}
            
                \subsection{Related Notions}
                \label{subsec:uni}
            
                We discuss some options other than \eqref{eq:intro_bgd1} for measuring data heterogeneity \cite{wang2022unreasonable,Li2020On,pu2021sharp,huang2022tackling,anonymous2023linear,malinovsky2022server,huang2021distributed} in this part for comparison purpose. Generally speaking, the measurement \eqref{eq:NonIID_measure} can be applied to general smooth nonconvex objective functions, while the related works only consider smooth strongly convex objective functions \cite{wang2022unreasonable,Li2020On,pu2021sharp,huang2022tackling}, nonconvex objective functions satisfying the Polyak-Łojasiewic (PL) condition, or a different sampling strategy, random reshuffling (RR) \cite{huang2021distributed,malinovsky2022server}.
                
                The works in \cite{Li2020On,anonymous2023linear,malinovsky2022server,huang2021distributed} consider the same measurement for data heterogeneity as \eqref{eq:NonIID_measure}. However, the work in \cite{Li2020On} requires each $f_i$ to be smooth strongly convex; the author in \cite{anonymous2023linear} assumes the interpolation regime where each $f_i$ satisfies the PL condition with $\uf_i = \uf$; and the results in \cite{malinovsky2022server,huang2021distributed} consider the RR-type sampling strategy. The proposed \eqref{eq:NonIID_measure} can be seen as a generalization of these previous results. Moreover, as we only require Assumption \ref{as:abc} instead of the uniformly bounded stochastic gradient assumption in \cite{Li2020On} and Assumption \ref{as:smooth} instead of the PL condition and other restrictive assumptions in \cite{anonymous2023linear}, such a generalization is nontrivial. Our general settings coincide with a wide range of practical problems. 
            
                The papers \cite{pu2021sharp,huang2022tackling,khaled2020tighter,koloskova2020unified} consider the term $
                1/n\sumn\normi{\nabla f_i(x^*)}^2$ to measure data heterogeneity when each $f_i$ is smooth (strongly) convex. Lemma \ref{lem:equ} shows that such an option is equivalent to \eqref{eq:bgd1} when the PL condition is satisfied; {note that PL condition is implied by strong convexity \cite{karimi2016linear}}. 
            
                \begin{lemma}
                    \label{lem:equ}
                    Let Assumption \ref{as:smooth} hold and each $f_i$ satisfies the PL condition, i.e., $\frac{1}{2}\norm{\nabla f_i(x)}^2\geq \mu (f_i(x) - f_i(x_{i}^*))$ with $x_i^* = \arg\min_x f_i(x)$. Then,
                    \begin{align*}
                        2\mu\prt{f(x) -  \uf} + 2\mu{\sigfn}
                        &\leq \frac{1}{n}\sumn \norm{\nabla f_i(x)}^2\\
                        &\leq 2L\prt{f(x) - \uf} + 2L{\sigfn}.
                    \end{align*}
                \end{lemma}
            
                \begin{proof}
                    We have $\uf_i = f_i(x_i^*)$, and hence 
                    \begin{align*}
                        2\mu\prt{f_i(x) - f_i^*}\leq \norm{\nabla f_i(x)}^2\leq 2L\prt{f_i(x) - \uf_i}.
                    \end{align*}
                    Averaging among $i\in[n]$ yields the desired result.
                \end{proof}
            
                The paper \cite{wang2022unreasonable} proposes a new notion $\rho$ defined in \eqref{eq:rho} to characterize the effect of data heterogeneity tailored for FedAvg when each component function is smooth and strongly convex,
                \begin{equation}
                    \label{eq:rho}
                    \rho:= \norm{\cG(x^*)} = \norm{\frac{1}{n}\sumn \brk{\frac{1}{\eta Q}\prt{x^* - x_i^{(Q)}} }},
                \end{equation}
                where $x^*$ is the unique solution to problem \eqref{eq:P}, and $x_i^{(Q)}$ denotes the iterate of agent $i$ performing $Q$ local updates with stepsize $\eta$ starting from $x^*$. Such a measurement $\rho$ can be seen as a more accurate version of the term $1/n\sumn\normi{\nabla f_i(x^*)}^2$, which is a conservative upper bound of $\rho^2$. Although it was shown to be a good measurement of the data heterogeneity in Section 4 of \cite{wang2022unreasonable}, the optimal solution $x^*$ becomes ambiguous in the general nonconvex world. 
            
                \section{Convergence Analysis for FedAvg}
                \label{sec:fedavg}
            
                In this section, we introduce the convergence result for FedAvg (Algorithm \ref{alg:fedavg}) under only Assumptions \ref{as:abc} and \ref{as:smooth}. We consider the simplified setting with full user participation, equal number of local updates, and equal weighted objective functions. 
            
                \begin{algorithm}
                    \caption{FedAvg}
                    \label{alg:fedavg}
                    \begin{algorithmic}[1]
                        \STATE {\textbf{Server input: } initial $x_0$ and number of server updates $T$}.
                        \STATE {\textbf{Agent input: } number of local updates $Q$ and stepsize policy $\crk{\alpha_t}$}.
                        \FOR{$t=0,1,\cdots, T-1$}
                            \FOR{Agent $i\in[n]$ in parallel}
                                \STATE Receive $x_t$ from server and set $x_{i,t}^0 = x_t$.
                                 \FOR{$\ell = 0,1,\cdots, Q-1$}
                                    \STATE Query a stochastic gradient $g_i(x_{i,t}^\ell;\xi_{i,t}^\ell)$.
                                    \STATE Update $x_{i,t}^{\ell + 1} = x_{i,t}^{\ell} - \alpha_t g_i(x_{i,t}^\ell;\xi_{i,t}^\ell)$.
                                \ENDFOR 
                            \ENDFOR
                            \STATE Server aggregates $x_{t + 1} = \frac{1}{n}\sumn x_{i,t}^Q$ and sends $x_{t + 1}$ to all the agents.
                        \ENDFOR 
                        \STATE Output $x_{T}$.
                    \end{algorithmic}
                \end{algorithm}
            
                Algorithm \ref{alg:fedavg} yields the update:
                \begin{align}
                    \label{eq:xt}
                    x_{t + 1} = x_t - \frac{\alpha_t}{n}\sumn \sum_{\ell=0}^{Q-1} g_i(\xitl;\xi_{i,t}^\ell).
                \end{align}
            
                The synchronization at the beginning of each iteration $t$ gives rise to $x_{i,t}^0 = x_t, \forall t\geq 0$. Such a relation plays a major role in the analysis, which indicates that
                \begin{equation}
                    \label{eq:xitl_xt}
                    \begin{aligned}
                        x_{i,t}^{\ell + 1} &= \xitl - \alpha_t g_i(\xitl;\xi_{i,t}^\ell) =  x_t - \alpha_t \sum_{j=0}^\ell g_i(x_{i,t}^j;\xi_{i,t}^j),\\
                        & \ell \in \crk{0, 1,\cdots, Q-1}, \ t\geq 0.
                    \end{aligned}
                \end{equation}

                \subsection{Convergence Results: FedAvg}
                
                We introduce the main convergence results in this part. Besides the usual complexity result stated in Theorem \ref{thm:complexity}, when the stepsize policy satisfies a certain diminishing condition, we also present the asymptotic convergence result in Theorem \ref{thm:asy} in light of Theorem 2.1 in \cite{li2022unified}. The latter result complements the first one and provides partial guarantees that choosing the last iterate as the output is reasonable in practice.
                
                \begin{theorem}
                    \label{thm:complexity}
                    Suppose Assumptions \ref{as:abc} and \ref{as:smooth} hold, and let the stepsize $\alpha_t= \alpha$ satisfy 
                    
                    \begin{equation}
                        \label{eq:alphat}
                        \begin{aligned}
                            \alpha\leq \min&\crk{\sqrt{\frac{n}{Q(C^2 + L^2) T}}, \prt{\frac{1}{14Q^3L^2(2C+3L)T}}^{\frac{1}{3}}\right.\\
                            &\left.\quad \frac{1}{C}, \frac{1}{2\sqrt{2QCL}}},\; \forall t.
                        \end{aligned}
                    \end{equation}
                    
                    Then,
                    
                    \begin{equation}
                        \label{eq:minE}
                        \begin{aligned}
                            &\min_{t=0,1,\cdots,T-1} \E\brk{\norm{\nabla f(x_t)}^2} \leq \frac{12\prt{f(x_0) - \uf}}{\alpha Q T}\\
                            &\quad + \frac{2\alpha L}{n}\prt{D + 2C\sigfn}+ 4\alpha^2QL^2D\\
                            &\quad + 8\alpha^2L^2Q^2(3L+2C)\sigfn.
                        \end{aligned}
                    \end{equation}

                    In addition, if we set 
                    \begin{equation}
                        \label{eq:alphat_gamma}
                        \begin{aligned}
                            \alpha& = \frac{1}{\sqrt{\frac{Q(L^2+C^2)T}{n}} + \gamma},\\
                            \gamma&:= C + \prt{14Q^3L^2(2C+3L)T}^{1/3} + 2\sqrt{2QCL},\; \forall t,\\
                        \end{aligned}
                    \end{equation}
                    then 
                    \begin{equation}
                        \label{eq:minE_alpha}
                        \begin{aligned}
                            &\min_{t=0,1,\cdots,T-1} \E\brk{\norm{\nabla f(x_t)}^2} = \order{\frac{D + C\prt{\uf - \sumn\uf_i/n}}{\sqrt{nQT}}\right.\\
                            &\left.\quad + \frac{nD}{T}+ \frac{(1+C)\prt{\uf- \sumn\uf_i/n}}{T^{2/3}}}.
                        \end{aligned}
                    \end{equation}

                \end{theorem}

                \begin{proof}
                    See Appendix \ref{app:thm_complexity}.
                \end{proof}

                \begin{remark}
                    {We highlight the technical challenge of obtaining Theorem \ref{thm:complexity} compared with the earlier works that consider Assumption \eqref{as:abc} and condition \eqref{eq:intro_bv1}. The main difference results from the distributed feature of FedAvg, that is, each agent maintains its own iterate $x_{i,t}^\ell$ during the local update phases. Such a feature introduces the additional term $\sum_{\ell = 0}^{Q-1}\sumn\E[f_i(\xitl) - \uf_i|\cF_t]$ under Assumption \ref{as:abc}. Hence, additional analysis is needed to relate such a term to the function value at the server side, i.e., $(\E f(x_t) - \uf)$ compared to those in \cite{li2022unified,lei2019stochastic,khaled2020better}. The difference between the server's iterate and the agents' local iterates also imposes additional difficulties since it is coupled with the recursion of $\sum_{\ell = 0}^{Q-1}\sumn\E[f_i(\xitl) - \uf_i|\cF_t]$. This is different from those considering FedAvg with condition \eqref{eq:intro_bv1}, e.g., \cite{yu2019linear,wang2020tackling,yang2021achieving}}. 
                \end{remark}
            
                \begin{remark}
            The convergence results given in Theorem \ref{thm:complexity} coincide with those in earlier works up to constant factors; see, e.g., \cite{yu2019linear,wang2020tackling,yang2021achieving}. 
            \end{remark}
            
                For comparison purpose, we introduce Corollary \ref{cor:Ceq0} which states the convergence result of FedAvg when assuming uniformly bounded variance on the stochastic gradients, i.e., when $C=0$. 
                \begin{corollary}
                    \label{cor:Ceq0}
                    Suppose Assumption \ref{as:abc} holds with $C=0$ and Assumptions \ref{as:smooth} holds. Let the stepsize $\alpha_t= \alpha$ satisfy
                    \begin{align*}
                        \alpha\leq \min&\crk{\sqrt{\frac{n}{QL^2 T}}, \prt{\frac{1}{42Q^3L^3T}}^{\frac{1}{3}}},\; \forall t.
                    \end{align*}
                    Then, 
                    \begin{align*}
                        &\min_{t=0,1,\cdots,T-1} \E\brk{\norm{\nabla f(x_t)}^2} \leq \frac{12\prt{f(x_0) - \uf}}{\alpha Q T}\\
                        &\quad + \frac{2\alpha L D}{n}+ 4\alpha^2QL^2D + 24\alpha^2L^3Q^2\sigfn.
                    \end{align*}
            
                    In addition, if we set $\alpha = 1/[\sqrt{QL^2T/n} + (42T)^{1/3}QL],\;\forall t$,
                    then 
                    \begin{align*}
                        &\min_{t=0,1,\cdots,T-1} \E\brk{\norm{\nabla f(x_t)}^2}\\
                        & = \order{ \frac{D}{\sqrt{nQT}}  + \frac{nD}{T} + \frac{\uf- \sumn\uf_i/n}{T^{2/3}}}.
                    \end{align*}
                    
                \end{corollary}
             The result of Corollary \ref{cor:Ceq0} directly follows from the derivation of Theorem \ref{thm:complexity} by taking $C=0$. We can see that the influence of the term $(\uf - \frac{1}{n}\sumn \uf_i)$ measuring data heterogeneity disappears from the dominating term $\cO(1/\sqrt{nT})$. {Such an observation implies that assuming bounded variance on the stochastic gradients corresponds to an analysis that largely ignores data heterogeneity. By contrast, the more relaxed conditions specify the impact of data heterogeneity on the algorithmic performance and recovers the complexity of FedAvg when $C=0$.}
            
                The complexity results above assume a constant stepsize depending on $T$. When the stepsize policy satisfies a certain diminishing condition instead, Theorem \ref{thm:asy} states the asymptotic convergence result of $\normi{\nabla f(x_t)}$ and $\E[\normi{\nabla f(x_t)}]$, which serves as a complement of the complexity result in Theorem \ref{thm:complexity}. It provides a guarantee for FedAvg to output the last iterate for large $T$. 
            
                \begin{theorem}
                    \label{thm:asy}
                    Let Assumptions \ref{as:abc} and \ref{as:smooth} hold, and let the stepsize policy $\crk{\alpha_t}_{t\geq 0}$ satisfy
                    \begin{align*}
                        &\alpha_t\leq \min\crk{\frac{1}{C},\frac{1}{2QL\sqrt{3}}, \frac{1}{2\sqrt{2QCL}}},\; \forall t\\
                        &\sum_{t=0}^{\infty}\alpha_t = \infty, \ \sum_{t=0}^\infty \alpha_t^2 <\infty.
                    \end{align*}
                    Then, $\lim_{t\rightarrow\infty}\E[\normi{\nabla f(x_t)}] = 0$, and $\lim_{t\rightarrow\infty}\normi{\nabla f(x_t)} = 0$ almost surely.
                \end{theorem}
                \begin{proof}
                    The detailed proof can be found in Appendix \ref{app:thm_asy}. The proof is to verify the conditions of Theorem 2.1 in \cite{li2022unified} by noting:
                    \begin{equation}
                        \label{eq:xtd}
                        \begin{aligned}
                        &\E\brk{\norm{x_{t + 1} - x_t}^2} \leq \frac{\alpha_t^2QC}{n}\sum_{\ell=0}^{Q-1}\sumn \E\brk{f_i(\xitl) - \uf_i}\\
                        &\quad + \frac{\alpha_t^2 Q}{n}\sum_{\ell=0}^{Q-1}\sumn \E\brk{\norm{\nabla f_i(\xitl)}^2} + \alpha_t^2Q^2 D.
                    \end{aligned}
                    \end{equation}
                \end{proof}

                The remaining parts of this section introduce the supporting lemmas for proving Theorem \ref{thm:complexity} and Theorem \ref{thm:asy}. In particular, we start with Lemma \ref{lem:descent} which provides the intuition on why considering bounded variance may partially neglect the influence of data heterogeneity. Lemma \ref{lem:descent} also guides us to consider the recursions of the corresponding terms in Subsection \ref{subsec:supporting}.
            
                \subsection{Descent Lemma and Data Heterogeneity}
                \label{subsec:descent}
                
                Lemma \ref{lem:descent} shows the one-step improvement of the function value in light of the descent lemma and Assumption \ref{as:abc}. 
                 \begin{lemma}
                    \label{lem:descent}
                    Suppose Assumptions \ref{as:abc} and \ref{as:smooth} hold. Let $\alpha_t\leq 1/(QL)$. Then,
                    \begin{align}
                        &\condE{f(x_{t + 1}) - \uf}{\cF_t} \leq f(x_t) - \uf - \frac{\alpha_tQ}{2}\norm{\nabla f(x_t)}^2 \label{inequality:descent}\\
                        &\quad + \frac{\alpha_t L^2}{2n}\sum_{\ell=0}^{Q-1}\sumn\condE{\norm{x_t - \xitl}^2}{\cF_t} \nonumber\\
                        &\quad + \frac{\alpha_t^2 CL}{2n^2}\sum_{\ell=0}^{Q-1}\sumn \condE{f_i(\xitl) - \uf_i}{\cF_t} + \frac{\alpha_t^2L QD}{2n}.\nonumber
                    \end{align}
                \end{lemma}
                 \begin{proof}
                    See Appendix \ref{app:lem_descent}.
                \end{proof}

             It can be seen from the above descent lemma that the performance of FedAvg is affected by the data heterogeneity among the agents.
             {Specifically, the data heterogeneity comes into effect thanks to two factors: a) the inconsistency among different agents due to local updates; b) the noisy gradients of different agents.
                On one hand, note that when $Q>1$, the term $\E[\normi{x_t - x_{i,t}^\ell}^2| \cF_t]$ in \eqref{inequality:descent} is active. In light of Lemma \ref{lem:descent} and \eqref{eq:xitl_xt}, the upper bound on $\E[\normi{x_t - x_{i,t}^\ell}^2| \cF_t]$ relates to $\sumn\normi{\nabla f_i(\xitl)}^2$, which is non-zero even if $\xitl=x^*,\forall i$ because of data heterogeneity. This is the case even when full gradients are available (see Lemma \ref{lem:local_x} for details). 
                On the other hand, when $C>0$ which implies noisy gradients, the term $\sumn \condEi{f_i(\xitl) - \uf_i}{\cF_t}$ is active which also relates to the measure of data heterogeneity. Such a result also explains why assuming $C=0$ in Assumption \ref{as:abc} or $\eta^2 =0$ in condition \eqref{eq:intro_bv1} is not preferred since partial influence of the data heterogeneity would be ignored.}
            
                \begin{remark}
                    The last two terms in Lemma \ref{lem:descent} are related to Assumption \ref{as:abc}. By comparison, the more restrictive condition \eqref{eq:intro_bv1} yields the term $\frac{\alpha_t^2 L\eta^2 }{2n^2}\sum_{\ell=0}^{Q-1}\sum_{i=1}^n\normi{\nabla f_i(\xitl)}^2 + \frac{\alpha_t^2L Q\sigma^2}{2n}$. 
                \end{remark}

                \subsection{Supporting Lemmas}
                \label{subsec:supporting}
                
                In this section, we present the supporting lemmas that enable us to derive the complexity result for FedAvg in Theorem \ref{thm:complexity} and establish the almost sure convergence in Theorem \ref{thm:asy}. In lighf of Lemma \ref{lem:descent}, we derive Lemma \ref{lem:local_fi} and Lemma \ref{lem:local_x} to estimate the corresponding terms in Lemma \ref{lem:descent}. It is noteworthy that the result of Lemma \ref{lem:local_fi} comes from the ABC assumption \ref{as:abc} under the distributed setting and imposes additional technical challenges to the analysis.  

                By combining Lemmas \ref{lem:descent}-\ref{lem:local_x}, we obtain Lemma \ref{lem:combined}, which constructs the celebrated recursion appearing in the Supermartingale Convergence Theorem \cite{ROBBINS1971233}. Lemma \ref{lem:combined} constitutes the key ingredient for deriving the complexity result as stated in Theorem \ref{thm:complexity}, in light of Lemma \ref{lem:min}. Furthermore, it serves as the starting point in establishing the almost sure convergence as stated in Theorem \ref{thm:asy}.
                
                We start from presenting Lemma \ref{lem:local_fi}, which is derived from the $L$-smoothness of $f_i$ and the relationship described in \eqref{eq:xitl_xt}.

                \begin{lemma}
                    \label{lem:local_fi}
                    Let Assumptions \ref{as:abc} and \ref{as:smooth} hold. 
                    We have 
                    \begin{equation}
                        \label{eq:local_fi}
                        \begin{aligned}
                            &\sumn\sum_{\ell=0}^{Q-1}\brk{f_i(x_{i,t}^\ell) - \uf_i} \leq 2nQ\brk{f(x_t) - f^*}\\
                            &\quad + 2nQ\sigfn + L\sumn\sum_{\ell=0}^{Q-1}\norm{x_{i,t}^\ell - x_t}^2.
                        \end{aligned}
                    \end{equation}
                \end{lemma}
                \begin{proof}
                    See Appendix \ref{app:lem_local_fi}
                \end{proof}
                
                Then, Lemma \ref{lem:local_x} is derived by noting the relationship stated in \eqref{eq:xitl_xt}.
                \begin{lemma}
                    \label{lem:local_x}
                    Let Assumptions \ref{as:abc} and \ref{as:smooth} hold. Set the stepsize $\alpha_t \leq \min\crki{1/(QL\sqrt{12}), 1/\sqrt{8QCL}},\;\forall t$. We have 
                    \begin{equation}
                        \label{eq:local_x}
                        \begin{aligned}
                            &\sum_{\ell=0}^{Q-1}\sumn \condE{\norm{x_t - \xitl}^2}{\cF_t}\leq 6\alpha_t^2 Q^3 n\norm{\nabla f(x_t)}^2\\
                            &\quad + 4\alpha_t^2Q^3n(2C+3L)\brk{f(x_t) - \uf}+ 2\alpha_t^2Q^2nD \\
                            &\quad + 4\alpha_t^2Q^3n(3L+2C)\sigfn.
                        \end{aligned}
                    \end{equation}
                    
                \end{lemma}
                \begin{proof}
                    See Appendix \ref{app:lem_local_x}.
                \end{proof}

                By combining Lemmas \ref{lem:descent}-\ref{lem:local_x}, we are able to derive the recursion of $\condEi{f(x_{t + 1}) - f^*}{\cF_t}$, as described in Lemma \ref{lem:combined}.

                \begin{lemma}
                    \label{lem:combined}
                    Let Assumptions \ref{as:abc} and \ref{as:smooth} hold. Set the stepsize $\alpha_t$ to satisfy 
                    \begin{align*}
                        \alpha_t\leq \min\crk{\frac{1}{C},\frac{1}{2QL\sqrt{3}}, \frac{1}{2\sqrt{2QCL}}, \frac{1}{QL}}, \; \forall t.
                    \end{align*}
                    
                    Then,
                    \begin{equation}
                        \label{eq:re_f}
                    \begin{aligned}
                        &\condE{f(x_{t + 1}) - \uf}{\cF_t}\leq \prt{\frac{1}{2nQ} + \alpha_t L}\alpha_t^2 Q^2 LD \\
                        & + \brk{1 + 2\alpha_t^3Q^3L^2(2C+3L)+ \frac{\alpha_t^2 Q CL}{n}}\brk{f(x_t) - \uf}\\
                        & + \alpha_t^2 LQ^2\brk{\frac{C}{nQ} + 2\alpha_tQL(3L+2C)}\sigfn - \frac{\alpha_tQ}{4}\norm{\nabla f(x_t)}^2.
                    \end{aligned}
                    \end{equation}
                    
                \end{lemma}
                \begin{proof}
                    Let $\alpha_t\leq 1/C$. Substituting \eqref{eq:local_fi} and \eqref{eq:local_x} into \eqref{inequality:descent} leads to the desired result.
                \end{proof}

                \begin{lemma}[Lemma 6, \cite{mishchenko2020random}]
                    \label{lem:min}
                    Suppose that there exist constants $a, b, c\geq 0$ and nonnegative sequences $\crk{s_t}_{t=0}^T$, $\crk{q_t}_{t=0}^T$ such that for any $t$ satisfying $0\leq t\leq T$, we have the recursion $s_{t + 1}\leq (1 + a)s_t -bq_t + c$.
                    Then, the following holds: 
                    \begin{align*}
                        \min_{t=0,1,\cdots, T-1} q_t \leq \frac{(1 + a)^T}{bT}s_0 + \frac{c}{b}.
                    \end{align*}
                \end{lemma}

\section{Convergence Analysis for SCAFFOLD}

In this section, we present the convergence result for SCAFFOLD\footnote{We assume the full participation of all the agents for simplification.} (Algorithm \ref{alg:scaffold}) under Assumptions \ref{as:abc} and \ref{as:smooth}, i.e., the same setting as in Section \ref{sec:fedavg}. SCAFFOLD is known to be able to mitigate the influence of data heterogeneity more effectively than FedAvg. This is demonstrated by the absence of the parameters $\zeta^2$ and $\psi^2$ (in the BGD condition \eqref{eq:intro_bgd1}) in the convergence result stated in \cite{karimireddy2020scaffold}. 

The convergence results for SCAFFOLD given in this section inherit such a characteristic, avoiding an additional term involving $(\uf - \sumn\uf_i / n)$ when compared to the convergence result of FedAvg in Theorem \ref{thm:complexity}. However, the term $C(\uf - \sumn\uf_i / n)$ persists due to Assumption \ref{as:abc}. 
As long as $C>0$, the impact of data heterogeneity remains.
Such a finding is consistent with the results of FedAvg and aligns with the arguments in subsection \ref{subsec:descent}, emphasizing that considering Assumption \ref{as:abc} allows for characterizing the partial impact of data heterogeneity due to stochastic gradients. By contrast, the conclusions given in \cite{karimireddy2020scaffold} do not reflect the influence of data heterogeneity on the algorithm.

\begin{algorithm}
    \caption{SCAFFOLD}
    \label{alg:scaffold}
    \begin{algorithmic}[1]
        \STATE \textbf{Server input: } initial $x_0$, number of server updates $T$, and stepsize $\eta_s$.
        \STATE \textbf{Agent input: } initial $c_{i,0}$, number of local updates $Q$, and stepsize $\eta_a$.
        \STATE Server initializes $c_0 = \sumn c_{i,0} / n$.
        \FOR{$t=0,1,\cdots, T-1$}
            \FOR{Agent $i\in[n]$ in parallel}
                \STATE Receive $x_t$ from server and set $x_{i,t}^0 = x_t$.
                 \FOR{$\ell = 0,1,\cdots, Q-1$}
                    \STATE Query a stochastic gradient $g_i(x_{i,t}^\ell;\xi_{i,t}^\ell)$.
                    \STATE Update $x_{i,t}^{\ell + 1} = x_{i,t}^{\ell} - \eta_a (g_i(x_{i,t}^\ell;\xi_{i,t}^\ell) - c_{i,t} + c_t)$.
                \ENDFOR 
                \STATE Update $c_{i, t + 1} = c_{i,t} - c_t + \frac{1}{\eta_a Q}\prt{x_t - x_{i,t}^Q}$.
            \ENDFOR
            \STATE Server updates $x_{t + 1} = x_t + \frac{\eta_s}{n}\sumn(x_{i,t}^Q - x_t)$ and $c_{t + 1} = c_t + \frac{1}{n}\sumn\prt{c_{i, t + 1} - c_{i, t}}$. \label{line:scaffold_agg}
            \STATE Server sends $(x_{t + 1}, c_{t + 1})$ to all the agents.
        \ENDFOR 
        \STATE Output $x_{T}$.
    \end{algorithmic}
\end{algorithm}

Referring to Line \ref{line:scaffold_agg} and the initialization $c_0 = \sumn c_{i,0} / n$ in Algorithm \ref{alg:scaffold}, we deduce the following update:
    \begin{equation}
        \label{eq:scaffold_ct}
        \begin{aligned}
            c_{t + 1} 
            &= \frac{1}{n}\sumn c_{i,t + 1}.
        \end{aligned}
    \end{equation}
    
As a consequence, the following update applies to $\crki{x_t}$ similar to \eqref{eq:xt}:
\begin{align*}
    x_{t + 1} 
    &= x_t - \frac{\eta_s\eta_a}{n}\sum_{i=1}^n\sum_{\ell=0}^{Q-1}g_i(x_{i,t}^\ell;\xi_{i,t}^\ell):= x_t + \Delta x_t.
\end{align*}

Therefore, the proof for SCAFFOLD follows similar procedures as those in Section \ref{sec:fedavg} and is moved to Appendix \ref{app:scaffold} for conciseness. Compared with the analysis for FedAvg, the primary challenge lies in dealing with the additional parameters $c_{i,t}$ and $c_t$. Such a challenge is addressed by locating a Lyapunov function $\cL_k$:
\begin{equation}
    \label{eq:scaffold_lya_full}
    \begin{aligned}
        \cL_k &= \E\brk{f(x_t) - f^*} + \frac{2\teta L^2}{nQ}\sumn\sum_{\ell=0}^{Q-1}\E\brk{\norm{x_{i,t}^\ell - x_t}^2} \\
            &\quad +  \frac{56\teta^3 L^4}{\eta_s^2}\E\brk{\norm{\Delta x_t}^2},
    \end{aligned}
\end{equation}
where $\teta:= \eta_a\teta_s Q$ denotes the effective stepsize.

\subsection{Convergence Results: SCAFFOLD} 

{The complexity result for SCAFFOLD under Assumption \ref{as:abc} and Assumption \ref{as:smooth} is outlined in Theorem \ref{thm:scaffold}. In addition, the complexity result for the bounded variance case ($C=0$) is presented in Corollary \ref{cor:scaffold}. 

\begin{theorem}
    \label{thm:scaffold}
    Let Assumptions \ref{as:abc} and \ref{as:smooth} hold. Set the effective stepsize $\teta$ to satisfy
   \begin{equation}
    \label{eq:scaffold_teta}
     \begin{aligned}
         \teta\leq \min&\crk{\frac{\eta_s}{\sqrt{84 L(L+C)}}, \frac{1}{12(L+C)}, \sqrt{\frac{2nQ}{(L^2 + C^2)T}},\right.\\
         &\quad\left. \prt{\frac{\eta_s^2 Q}{560 LC(L+C)T}}^{1/3}}.
     \end{aligned}
   \end{equation}
   Then,
    \begin{equation}
        \label{eq:scaffold_minE}
        \begin{aligned}
            &\min_{t=0,1,\ldots, T-1} \E\brk{\norm{\nabla f(x_t)}^2}\leq \frac{36\prt{f(x_0) - \uf}}{\teta T}\\
            &\quad + \frac{192\teta^2 L^2(\Delta_1 + \Delta_2)}{\eta_s^2 T}+ \frac{2\teta L(2C\sigfn + D)}{nQ}\\
            &\quad + \frac{(3L/T  + 14(L+C))64\teta^2 L(2C\sigfn + D)}{\eta^2_s Q},
        \end{aligned}
    \end{equation}
    where $\Delta_1:= \sumn\norm{c_{i,0} - c_0}^2 / n$ and $\Delta_2:= \sumn\norm{\nabla f_i(x_0)}^2/n$.

    In addition, if we set 
    \begin{equation}
        \label{eq:teta_etas}
        \begin{aligned}
            \teta &= \frac{1}{\sqrt{\frac{(L^2 +C^2) T}{2nQ}} + 12(L+C) + \gamma_s},\\
            \gamma_s &:= \frac{\sqrt{84 L(L+C)}}{\eta_s} +\prt{\frac{560LC(L+C)T}{\eta_s Q}}^{1/3},
        \end{aligned}
    \end{equation}
    then, 
    \begin{equation}
        \label{eq:scaffold_minE_etas}
        \begin{aligned}
            &\min_{t=0,1,\ldots, T-1} \E\brk{\norm{\nabla f(x_t)}^2}\\
            & = \order{\frac{D + C\prt{\uf - \sumn\uf_i/n}}{\sqrt{nQ T}} + \frac{C}{(\eta_sQ)^{1/3}T^{2/3}}\right.\\
            &\left.\quad + \frac{nD + nC\prt{\uf - \sumn\uf_i/n} }{\eta_s^2T}}.
        \end{aligned}
    \end{equation}
\end{theorem}

\begin{proof}
    See Appendix \ref{app:thm_scaffold}.
\end{proof}
}

Comparing the complexities of SCAFFOLD \eqref{eq:scaffold_minE_etas} and FedAvg \eqref{eq:minE_alpha} reveals that SCAFFOLD effectively mitigates the influence of data heterogeneity. Furthermore, when $\eta_s$ is set to $\orderi{\sqrt{n}}$ in \eqref{eq:scaffold_minE_etas}, SCAFFOLD improves the complexity of FedAvg with respect to the number of agents $n$. However, due to the stochastic gradients, data heterogeneity still impacts the algorithmic performance when $C>0$.

To further compare with the result in \cite{karimireddy2020scaffold}, we introduce Corollary \ref{cor:scaffold} which states the convergence result of SCAFFOLD following Theorem \ref{thm:scaffold} when assuming uniformly bounded variance on the stochastic gradients ($C = 0$).

\begin{corollary}
    \label{cor:scaffold}
    Suppose Assumption \ref{as:abc} holds with $C=0$ and Assumption \ref{as:smooth} holds. Let the effective stepsize $\teta$ satisfy 
    \begin{align*}
        \teta &= \frac{1}{\sqrt{\frac{L^2T}{2nQ}} + \gamma_s},\ \gamma_s := \frac{\sqrt{84 L^2}}{\eta_s} + \sqrt{70}L,
    \end{align*}
    then 
    \begin{align*}
        &\min_{t=0,1,\ldots, T-1} \E\brk{\norm{\nabla f(x_t)}^2}= \order{\frac{D}{\sqrt{nQ T}} + \frac{nD}{\eta_s^2T}}.
    \end{align*}
\end{corollary}

Corollary \ref{cor:scaffold} recovers the complexity of SCAFFOLD $\orderi{1/\sqrt{nQT} + 1/T}$ by setting $\eta_s = \orderi{\sqrt{n}}$ as in \cite{karimireddy2020scaffold}.

\section{Simulations}

In this section, we present two numerical experiments to demonstrate and complement the theoretical findings. Specifically, the first example, as described in \eqref{eq:f_ex}, shows the effectiveness of FedAvg and  SCAFFOLD for dealing with stochastic gradients that satisfy the ABC assumption. The second example in \eqref{eq:invariant}, illustrates that the measurement $(f^* - \frac{1}{n}\sumn \uf_i)$ is not a trivial upper bound. Such a measurement changes with varying degrees of data heterogeneity and directly affects the algorithmic performance. By contrast, the constants $\zeta^2$ and $\psi^2$ in the BGD condition \eqref{eq:intro_bgd1} may not reflect the convergence properties of FedAvg and  SCAFFOLD, as also shown in this example. All the experimental results are presented by averaging over $100$ repeated runs.

\subsection{Performance of FedAvg and SCAFFOLD under the General Variance Condition}

We begin by verifying the effectiveness
of FedAvg and SCAFFOLD under stochastic gradients that violate the relaxed growth condition \eqref{eq:intro_bv1} but satisfy the ABC condition \eqref{eq:abc}. To show this, we examine a heterogeneous data setting with $n=16$ in which $f_1(x)$ and $f_i(x), i = 2,3,\cdots, 16$ are defined as follows: 
\begin{equation}
    \label{eq:f_ex} 
    \begin{aligned}    
        f_1(x) &= \begin{cases}     \frac{x^2}{2}, & |x|< 1 \\     
        |x| - \frac{1}{2}, & \text{otherwise}     
    \end{cases},\\
    f_i(x) &= \ln(1 + e^{(x - i + 1)}), i =2,3,\cdots, 16.
    \end{aligned} 
\end{equation}
Note the function $\ln(1 + e^x)$ is the softplus activation functions \cite{glorot2011deep}. The stochastic gradients $g_1(x;\xi)$ and $g_i(x;\xi)$ for $f_1(x)$ and $f_i(x), i = 2,3,\cdots, 16$, respectively, are computed as follows:
\begin{equation}
    \label{eq:gs}
    \begin{aligned}     
        &g_1(x;\xi) = \begin{cases}     \nabla f_1(x) + \sqrt{|x|}, & \xi = 0\\     
        \nabla f_1(x) - \sqrt{|x|}, & \xi=1     
    \end{cases}, \\
    &g_i(x;\xi) = \begin{cases}
        \nabla f_i(x) + \sqrt{f_i(x)}, & \xi = 0\\
        \nabla f_i(x) - \sqrt{f_i(x)}, & \xi = 1
    \end{cases},i=2,3\cdots, 16.
    \end{aligned}
\end{equation}
The random variables $\xi$'s are independent among different $g_i$'s and follow a Bernoulli distribution with a probability of $p=1/2$. As shown in Appendix \ref{app:verify}, all the stochastic gradients defined in \eqref{eq:gs} violate the relaxed growth condition \eqref{eq:intro_bv1} but satisfy the ABC condition \eqref{eq:abc}.

It can be seen from Fig. \ref{fig:ex} that both FedAvg and SCAFFOLD converge to a neighborhood of the optimal value when the effective stepsize is properly selected, which demonstrates the effectiveness of FedAvg and SCAFFOLD. Moreover, SCAFFOLD can achieve better performance compared to FedAvg, which is consistent with the theoretical results.

\begin{figure}
    \centering
    \includegraphics[width=0.45\textwidth]{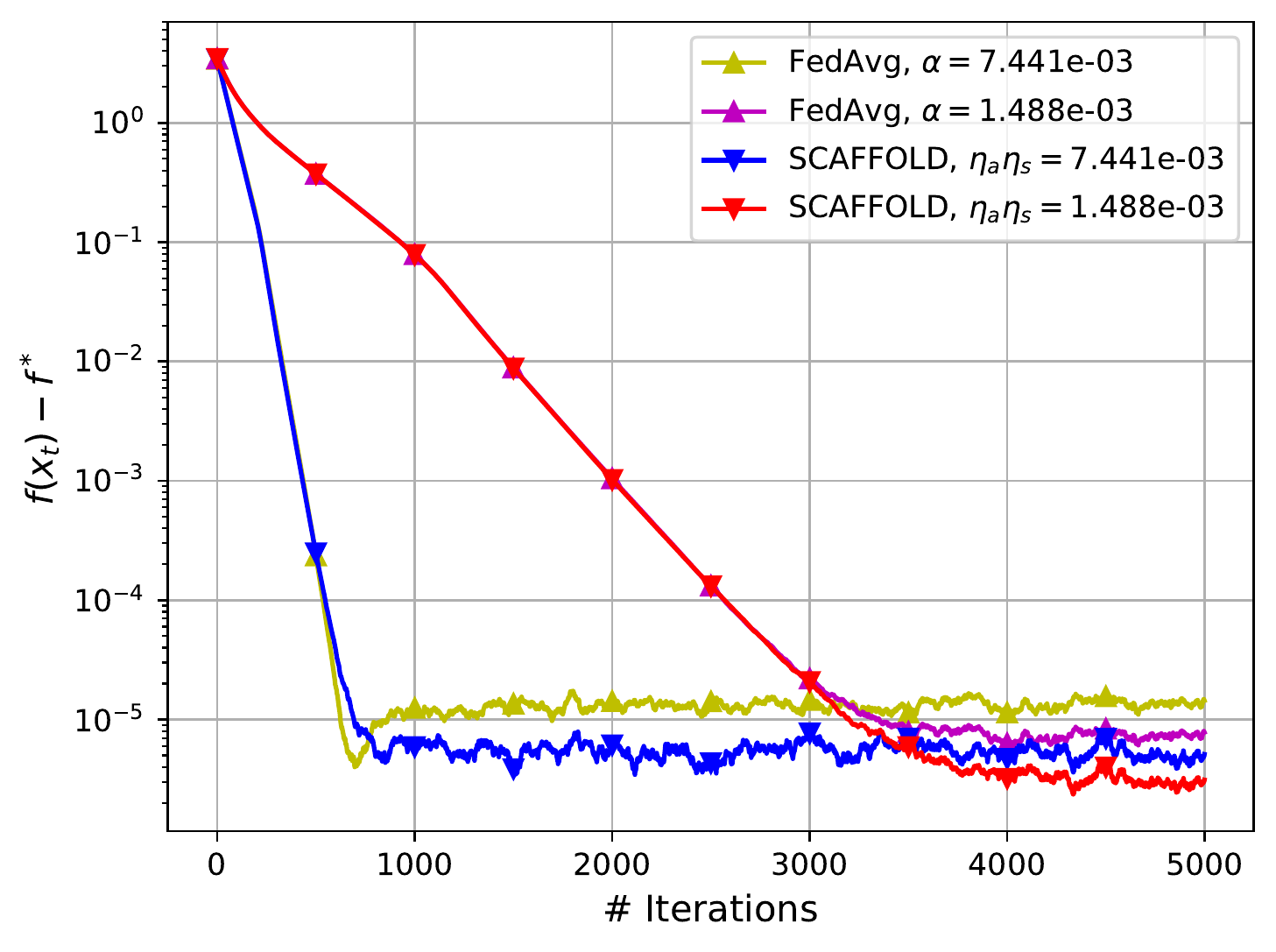}
    \caption{The performance of FedAvg and SCAFFOLD on example \eqref{eq:f_ex} with different stepsizes. The number of local updates is set as the same constant $17$.}
    \label{fig:ex}
\end{figure}

\subsection{Comparison with BGD Parameters}
\label{subsec:sim2}

In this second example, we consider the problem that satisfies the BGD condition \eqref{eq:intro_bgd1}, but the parameters $\zeta^2$ and $\psi^2$ are almost invariant to changes in data heterogeneity (different values of the parameter $|d|$'s), while the measurement $(\uf - \sumn\uf_i / n)$ is sensitive to such changes. 

Specifically, we consider the case where $n=2$, and the objective functions are defined as follows:
\begin{equation}
    \label{eq:invariant}
    \begin{aligned}
         f_1(x) &= x^2, \\
    f_2(x) &=
    \begin{cases}
    \frac{(x-d)^2}{2}, & |x-d|<1, \\
    |x-d| - \frac{1}{2}, & \text{otherwise}.
    \end{cases}
    \end{aligned}
\end{equation}
For this problem, larger values of $|d|$ correspond to  greater data heterogeneity and larger values of $f^* - \frac{1}{2}(f_1^* + f_2^*)$.

According to Table \ref{tab:invariant}, it is evident that $(\zeta^2 + \psi^2)$ remains nearly constant for different values of $d$, while $(\uf - \sumn\uf_i/n)$ varies considerably. \footnote{We obtain the minimal value of $(\zeta^2 + \psi^2)$ by solving the problem defined in \eqref{eq:subpro} approximately.} Furthermore, as depicted in Figure \ref{fig:invariant}, the performance of FedAvg and  SCAFFOLD changes with different values of $d$. This shows that $(\uf - \sumn\uf_i/n)$ is not a trivial upper bound and directly affects the performance of FedAvg and  SCAFFOLD. Table \ref{tab:invariant} corroborates the above argument. 
These observations are consistent with the argument that $(\uf - \sumn\uf_i/n)$ provides a better measure of data heterogeneity compared to the BGD condition.

\begin{table}[]
\setlength{\tabcolsep}{8pt}
\centering
\begin{tabular}{@{}ccccc@{}}
\toprule
$d$    & $\zeta^2+\psi^2$ & $f^* - \frac{1}{2}(f_1^* + f_2^*)$ & {FedAvg} & { SCAFFOLD} \\ \midrule
$-100$ & $5.194$          & $49.625$                           & {$250.478$} &{$307.168$}                             \\
$-50$  & $5.194$          & $24.625$                           & {$166.569$} & {$144.949$}                            \\
$-20$  & $5.194$          & $9.625$                            & {$47.099$} & {$49.656$}                              \\
$-2$   & $5.200$          & $0.625$                            & {$4.232$} & {$3.364$}                               \\ \bottomrule
\end{tabular}
\caption{Comparison of $\zeta^2+\psi^2$ and $f^* - \frac{1}{2}(f_1^* + f_2^*)$ for minimizing \eqref{eq:invariant} with varying $d$ values and the corresponding gaps $f(x_{4000}) - f^*$ {for FedAvg and SCAFFOLD. The stepsize is set as $0.00046$ for both methods.}}
\label{tab:invariant}
\end{table}

\begin{figure}
    \centering
    \includegraphics[width=0.45\textwidth]{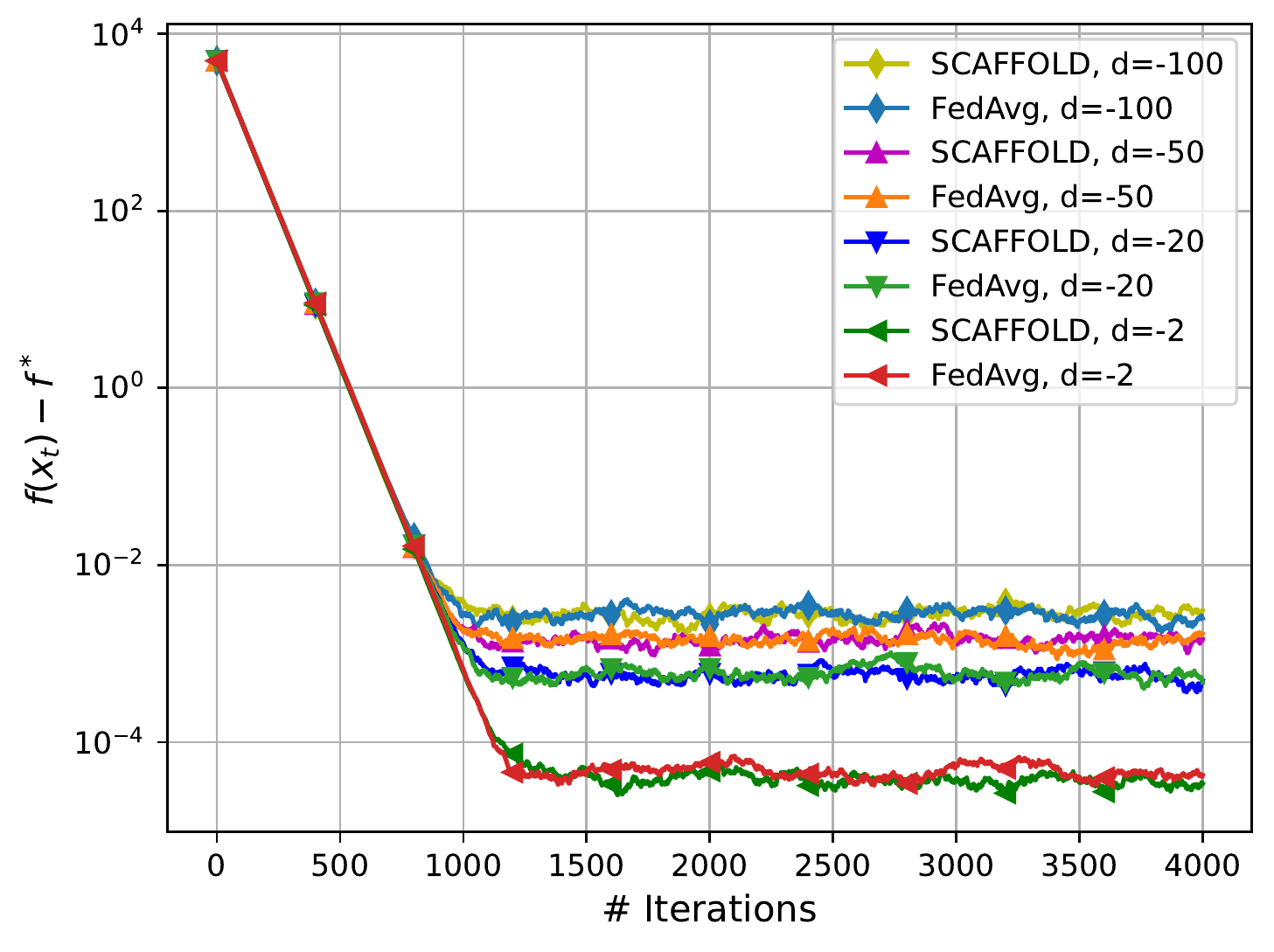}
    \caption{The performance of FedAvg and  SCAFFOLD for minimizing \eqref{eq:invariant} with different $d$'s. The stepsize is set as $0.00046$, and the number of local updates is set as $17$ for all $d$'s for fair comparison.}
    \label{fig:invariant}
\end{figure}

\begin{remark}
    We compare the potential approaches for measuring the degrees of data heterogeneity using the BGD condition in \eqref{eq:intro_bgd1} and the quantity $(\uf - \frac{1}{n}\sumn\uf_i).$ In many machine learning problems, $\uf_i = 0$, and accordingly, $f(x_T)$ can be utilized to estimate $(\uf - \frac{1}{n}\sumn\uf_i)$ after the training process. By contrast, obtaining meaningful values of $\zeta^2$ and $\psi^2$ in the BGD condition requires an additional subproblem to be solved. One possible solution is to solve the following subproblem, which is not trivial in general:
    \begin{equation}
        \label{eq:subpro}
        \begin{aligned}
        & \min_{\zeta,\psi >0} \zeta^2 + \psi^2,\\
        & \text{s.t.} \max_{x\in\R^p}\frac{1}{2}\sum_{i=1}^2 \Vert \nabla f_i(x) - \nabla f(x)\Vert^2 - \psi^2 \Vert \nabla f(x)\Vert^2 \leq \zeta^2.
        \end{aligned}
    \end{equation}
    Inaccurate estimates of $\zeta^2$ and $\psi^2$ cannot be used for measuring the degrees of data heterogeneity, since $\zeta^2$ and $\psi^2$ have no upper bounds. 
\end{remark}

\section{Conclusions}

    This paper is concerned with the key assumptions in distributed stochastic optimization, i.e., the assumption on the variance of the stochastic gradients and the bounded gradient dissimilarity (BGD) condition. By considering the more general ABC condition on the stochastic gradients and removing the BGD assumption, we show both FedAvg and SCAFFOLD maintain the same theoretical guarantee as the previous results. Moreover, the impact of data heterogeneity is clearly demonstrated in the complexity results of FedAvg and SCAFFOLD. The theoretical results are further supported with numerical experiments. Under diminishing stepsize, we establish almost sure convergence to a stationary point under the general condition for FedAvg.
    In addition, we discuss a nontrivial and informative measurement for data heterogeneity under general smooth nonconvex objective functions inspired by earlier works as well as its impacts on the algorithmic performance. 
    We anticipate that our analysis in this paper can also be applied to the decentralized setting, which is of future interest.

\appendices
\section{Proofs for FedAvg}
\label{app:part1}

\subsection{Proof of Lemma \ref{lem:ngrad}}
\label{app:lem_ngrad}
We have 
\begin{equation}
    \label{eq:uniform}
  \begin{aligned}
      &\frac{1}{n} \sumn{\norm{\nabla f_i(x) - \nabla f(x)}^2} \\
&= \frac{1}{n}\sumn\norm{\nabla f_i(x)}^2 + \norm{\nabla f(x)}^2 - \frac{2}{n}\sumn\inpro{\nabla f_i(x), \nabla f(x)}\\
&= \frac{1}{n}\sumn\norm{\nabla f_i(x)}^2 - \norm{\nabla f(x)}^2\\
&\leq \frac{1}{n}\sumn\norm{\nabla f_i(x)}^2\leq 2L\prt{f(x) - \frac{1}{n}\sumn \uf_i},
  \end{aligned}
\end{equation}
where the last inequality holds due to \eqref{eq:nf}.

\subsection{Proof of Lemma \ref{lem:h_con}}
\label{app:lem_h_con}
            We have 
            \begin{align}
                \norm{\nabla f_i(x) - \nabla f_i(x')}_2 &= \norm{\E_\xi \nabla F_i(x;\xi) - \E_\xi \nabla F_i(x';\xi)}_2\nonumber \\
                &\leq \E_\xi\norm{\nabla F_i(x;\xi) - \nabla F_i(x';\xi)}_2\label{eq:h_con_s1} \\
                &\leq \E_\xi L(\xi) \norm{x - x'}_2,\nonumber
            \end{align}
            where \eqref{eq:h_con_s1} holds due to Jensen's inequality \cite{rockafellar1997convex} that $\normi{\E X }_2\leq \E\normi{X}_2$ for a random variable $X$ ($\normi{\cdot}_2$ is convex).
            The remaining parts are obtained by invoking Proposition 3 in \cite{khaled2020better}.

\subsection{Proof of Lemma \ref{lem:descent}}
\label{app:lem_descent}

	Denote $g_{i,t}^\ell:= g_i(x_{i,t}^\ell;\xi_{i,t}^\ell)$. Applying descent lemma \eqref{eq:descent_lemma} to \eqref{eq:xt} and invoking the tower property yield
       
	\begin{equation}
		\label{eq:descent_s1}
		\begin{aligned}
			&\condE{f(x_{t + 1})}{\cF_t}\leq \frac{\alpha_t^2 L}{2}\condE{\norm{\frac{1}{n}\sumn\sum_{\ell=0}^{Q-1} g_{i,t}^\ell }^2}{\cF_t}\\
			&\quad + f(x_t) - \alpha_tQ \inpro{\nabla f(x_t), \frac{1}{nQ}\sumn\sum_{\ell=0}^{Q-1} \nabla f_i(\xitl)}\\
			&{\kh \leq} f(x_t)
			+ \frac{\alpha_tQ}{2}\condE{\norm{\nabla f(x_t) - \frac{1}{nQ}\sumn\sum_{\ell=0}^{Q-1} \nabla f_i(\xitl)}^2}{\cF_t}\\
            &\quad + \frac{\alpha_t^2 L}{2}\condE{\norm{\frac{1}{n}\sumn \sum_{\ell=0}^{Q-1}\brk{g_{i,t}^\ell - \nabla f_i(\xitl)}}^2}{\cF_t}\\
            &\quad- \frac{\alpha_tQ}{2}\norm{\nabla f(x_t)}^2,
		\end{aligned}
	\end{equation}
    where we let $\alpha_t \leq 1/(QL)$.
	We next consider the last term in \eqref{eq:descent_s1}. 
    
    \begin{align}
        &\condE{\norm{\frac{1}{n}\sum_{i=1}^n\sum_{\ell=0}^{Q-1} \brk{g_{i,t}^\ell - \nabla f_i(x_{i,t}^\ell)}}^2}{\cF_t}\nonumber\\
        &\quad = \frac{1}{n^2}\sumn\condE{\norm{\sum_{\ell=0}^{Q-1}\brk{g_{i,t}^\ell - \nabla f_i(x_{i,t}^\ell)}}^2}{\cF_t}\label{eq:var1_s1},
    \end{align}
    where the equality holds due to tower property and the unbiased property of $g_{i,t}^s$ in Assumption \ref{as:abc}. 
    We then consider the term in \eqref{eq:var1_s1}. 
    \begin{equation}
    \label{eq:var1_Q}
        \begin{aligned}
        &\condE{\norm{\sum_{\ell=0}^{Q-1}\brk{g_{i,t}^\ell - \nabla f_i(x_{i,t}^\ell)}}^2}{\cF_t}\\
        &= \sum_{\ell=0}^{Q-1}\condE{\norm{g_{i,t}^\ell - \nabla f_i(x_{i,t}^\ell)}^2}{\cF_t}\\
        &+ 2\sum_{0=r<s\leq Q-1}\condE{\inpro{g_{i,t}^r - \nabla f_i(x_{i,t}^r), g_{i,t}^s - \nabla f_i(x_{i,t}^s) }}{\cF_t}.
        \end{aligned}
    \end{equation}

    Due to $\cF_t\subset \cF_t^r\subset \cF_t^s$, we have from the tower property and $\condEi{g_{i,t}^s}{\cF_t^s} = \nabla f_i(x_{i,t}^s)$ in Assumption \ref{as:abc} that 
    \begin{equation}
        \label{eq:tower}
        \begin{aligned}
            &\condE{\inpro{g_{i,t}^r - \nabla f_i(x_{i,t}^r), g_{i,t}^s - \nabla f_i(x_{i,t}^s) }}{\cF_t}\\
            &= \condE{\condE{\inpro{g_{i,t}^r - \nabla f_i(x_{i,t}^r), g_{i,t}^s - \nabla f_i(x_{i,t}^s) }}{\cF_t^s}}{\cF_t} = 0.
        \end{aligned}
    \end{equation}

    Combing \eqref{eq:var1_s1}-\eqref{eq:tower} yields 
    \begin{equation}
        \label{eq:var_t}
        \begin{aligned}
            &\condE{\norm{\frac{1}{n}\sum_{i=1}^n\sum_{\ell=0}^{Q-1} \brk{g_{i,t}^\ell - \nabla f_i(x_{i,t}^\ell)}}^2}{\cF_t}\\
            &\leq \frac{C}{n^2}\sumn\sum_{\ell=0}^{Q-1}\condE{f_i(x_{i,t}^\ell) - \uf_i}{\cF_t} + \frac{Q D}{n}.
        \end{aligned}
    \end{equation}
    
	According to Assumption \ref{as:smooth}, we obtain
	\begin{equation}
		\label{eq:descent_diff}
		\begin{aligned}
			&\condE{\norm{\nabla f(x_t) - \frac{1}{nQ}\sumn\sum_{\ell=0}^{Q-1} \nabla f_i(\xitl)}^2}{\cF_t}\\
            & \leq \frac{L^2}{nQ}\sumn\sum_{\ell=0}^{Q-1}\condE{\norm{x_t - \xitl}^2}{\cF_t}.
		\end{aligned}
	\end{equation}

	We finish the proof by combing \eqref{eq:descent_s1}-\eqref{eq:descent_diff}.


\subsection{Proof of Lemma \ref{lem:local_fi}}
\label{app:lem_local_fi}


    We apply the descent lemma (stated in \eqref{eq:descent_lemma}) to $x_{i,t}^s$ and $x_t$ to get
    \begin{align}
        &f_i(x_{i,t}^s) - \uf_i \leq f_i(x_t) - \uf_i + \inpro{\nabla f_i(x_{t}), x_{i,t}^s - x_t}\nonumber\\
         &\quad + \frac{L}{2}\norm{x_{i,t}^s - x_t}^2\nonumber\\
         &\leq f_i(x_t) - \uf_i + \frac{1}{2L}\norm{\nabla f_i(x_t)}^2 + L\norm{x_{i,t}^s - x_t}^2\label{eq:local_fi_young}\\
         &\leq 2 \brk{f_i(x_t) - \uf_i} + L\norm{x_{i,t}^s - x_t}^2,\label{eq:local_fi_nf}
    \end{align}
    where \eqref{eq:local_fi_young} holds from Young's inequality that $\inpro{a,b}\leq \normi{a}^2/2/L + L\normi{b}^2/2$ and \eqref{eq:local_fi_nf} holds due to \eqref{eq:nf}. Note that the inequality \eqref{eq:local_fi_nf} holds for all $i\in[n]$ and $s\in\crki{0,1,\ldots, Q-1}$. Then, taking the summation over $i$ and $s$ in \eqref{eq:local_fi_nf} yields the desired result \eqref{eq:local_fi}.
    

\subsection{Proof of Lemma \ref{lem:local_x}}
\label{app:lem_local_x}

	Relation \eqref{eq:xitl_xt} yields $\xitl - x_t = -\alpha_t\sum_{j=0}^{\ell-1} g_i(x_{i,t}^j;\xi_{i,t}^j),\ell\in\crk{0,1,\cdots, Q-1}$. Recall that we denote $g_{i,t}^\ell:= g_{i}(x_{i,t}^\ell;\xi_{i,t}^\ell)$. Then,
    \begin{equation}
        \label{eq:xlt_split}
        \begin{aligned}
            &\xitl - x_t = -\alpha_t\sum_{j=0}^{\ell - 1}\crk{\brk{g_{i,t}^j - \nabla f_i(x_{i,t}^j)} + \nabla f_i(x_{i,t}^j)}.
        \end{aligned}
    \end{equation}

    Similar to \eqref{eq:var1_Q}, we have 
    \begin{equation}
        \label{eq:local_x_single}
        \begin{aligned}
            &\frac{1}{2\alpha_t^2}\condE{\norm{\xitl - x_t}^2}{\cF_t}\leq \sum_{j=0}^{\ell - 1}C\condE{f_i(x_{i,t}^j) - \uf_i}{\cF_t}\\
            &\quad + \ell D + \ell \sum_{j=0}^{\ell - 1}\condE{\norm{\nabla f_i(x_{i,t}^j)}^2}{\cF_t}.
        \end{aligned}
    \end{equation}

    Therefore, 
    \begin{equation}
        \label{eq:local_x_s1}
        \begin{aligned}
            &\sum_{\ell=0}^{Q-1}\sumn \condE{\norm{x_t - \xitl}^2}{\cF_t}\leq \alpha_t^2Q^2nD\\
            &\quad + 2\alpha_t^2QC\sumn\sum_{\ell = 0}^{Q-1}\condE{f_i(x_{i,t}^\ell) - \uf_i}{\cF_t}\\
            &\quad + \alpha_t^2 Q^2 \sumn\sum_{\ell =0}^{Q-1}\condE{\norm{\nabla f_i(x_{i,t}^\ell)}^2}{\cF_t}.
        \end{aligned}
    \end{equation}
    
    We next consider the last term in \eqref{eq:local_x_s1}. Note that 
    \begin{align*}
        \nabla f_i(\xitl) &= \nabla f_i(\xitl)-\nabla f_i(x_t) + \nabla f_i(x_t) - \nabla f(x_t)\\
        &\quad + \nabla f(x_t).
    \end{align*} 
    We have
	\begin{equation}
		\label{eq:local_x_s2}
		\begin{aligned}
			&\sum_{\ell=0}^{Q-1}\sumn \condE{\norm{\nabla f_i(\xitl)}^2}{\cF_t}
			\leq 3Qn\norm{\nabla f(x_t)}^2\\
            &\quad + 3L^2 \sum_{\ell=0}^{Q-1}\sumn\condE{\norm{x_{i,t}^\ell - x_t}^2}{\cF_t}\\
            &\quad + 6QnL\brk{\prt{f(x_t) - \uf} + {\kh \sigfn}},
		\end{aligned}
	\end{equation}
	where we invoke Lemma \ref{lem:ngrad} in the last inequality. Combing \eqref{eq:local_x_s1}, \eqref{eq:local_x_s2}, and \eqref{eq:local_fi} yields 
    \begin{equation}
        \label{eq:local_x_new}
        \begin{aligned}
            &\prt{1 - 3\alpha_t^2Q^2L^2 - 2\alpha_t^2 QCL}\sum_{\ell=0}^{Q-1}\sumn \condE{\norm{x_t - \xitl}^2}{\cF_t}\\
            &\leq  2\alpha_t^2Q^3n(2C+3L)\brk{f(x_t) - \uf}+ 3\alpha_t^2 Q^3 n\norm{\nabla f(x_t)}^2\\
            &\quad + \alpha_t^2Q^2nD + 2\alpha_t^2Q^3n(3L+2C)\sigfn.
        \end{aligned}
    \end{equation}

	Letting $\alpha_t \leq \min\crki{1/(QL\sqrt{12}), 1/(\sqrt{8QCL})}$ leads to the desired result \eqref{eq:local_x}.
	

\subsection{Proof of Theorem \ref{thm:complexity}}
\label{app:thm_complexity}

Let the stepsize $\alpha$ satisfy \eqref{eq:alphat}.
Then, we have 

\begin{equation}
    \label{eq:exp}
    \begin{aligned}
        &\brk{1 + 2\alpha^3Q^3L^2(2C+3L)+ \frac{\alpha^2 Q CL}{n}}^T\\
        &\leq \exp\crk{\brk{2\alpha^3Q^3L^2(2C+3L)+ \frac{\alpha^2 Q CL}{n}}T}\leq \exp(1).
    \end{aligned}
\end{equation}

Taking the full expectation on \eqref{eq:re_f} and applying Lemma \ref{lem:min} yield \eqref{eq:minE}. 
For $\alpha$ satisfying \eqref{eq:alphat_gamma}, we have 
\begin{equation}
    \label{eq:alpha_s1}
    \begin{aligned}
        &\frac{1}{\alpha QT} = \frac{C}{QT} +  \sqrt{\frac{C^2 + L^2}{nQT}} + \frac{\prt{14L^2(2C+3L)}^{1/3}}{T^{2/3}}\\
        &\quad +  \frac{2\sqrt{2CL}}{\sqrt{Q}T}.
    \end{aligned}
\end{equation}

Substituting \eqref{eq:alpha_s1} into \eqref{eq:minE} leads to \eqref{eq:minE_alpha}.

\subsection{Proof of Theorem \ref{thm:asy}}
\label{app:thm_asy}

	This is an application of Theorem 2.1 in \cite{li2022unified}.

	\begin{theorem}[Theorem 2.1 in \cite{li2022unified}]
		\label{thm:unified}
		Let the mapping $\Phi: \R^n \rightarrow \R^m$ and the sequences $\{x^k\}_{k\geq 0}\subseteq \R^n$ and $\{\mu_k\}_{k\geq 0} \subseteq \R_{++}$ be given. Consider the following conditions:    
	\begin{enumerate}[label=\textup{\textrm{(P.\arabic*)}},topsep=0pt,itemsep=0ex,partopsep=0ex]
		\item \label{P1} The function $\Phi$ is $L_{\Phi}$-Lipschitz continuous for some ${L}_{\Phi}>0$, i.e., we have $\|\Phi(x) - \Phi(y)\| \leq  {L}_{\Phi} \|x-y\|$ for all $x,y \in \R^n$. 
		\item \label{P2} There exists a constant $a > 0$ such that $\sum_{k=0}^{\infty} \ \mu_k \, \E[\|\Phi(x^k) \|^a]  < \infty$.
	\end{enumerate}
	The following statements are valid: 
	\begin{enumerate}[label=\textup{\textrm{(\roman*)}},topsep=0pt,itemsep=0ex,partopsep=0ex]
		\item Let the conditions \ref{P1}--\ref{P2} be satisfied and suppose further that
		\begin{enumerate}[label=\textup{\textrm{(P.\arabic*)}},topsep=0pt,itemsep=0ex,partopsep=0ex,start=3]
		\item \label{P3} there exist  constants ${A}, {B}, b \geq 0$ and $ p_1, p_2, q >0$ such that 
		\[ \E[\|x^{k+1} - x^k\|^q] \leq  {A} \mu_k^{p_1} \ + \ {B} \mu_k^{p_2}  \, \E[\|\Phi(x^k) \|^b], \]	 
		 \item \label{P4} the sequence $\{\mu_k\}_{k\geq 0}$  and the  parameters $a, b, q, p_1, p_2$ satisfy
		 \begin{align*}
            &\{\mu_k\}_{k\geq 0} \text{is bounded}, \quad {\sum}_{k=0}^{\infty} \ \mu_k = \infty,\\
            &\text{and} \quad a, q \geq 1, \;\; a \geq b,  \;\; p_1, p_2 \geq q. 
         \end{align*}
    	\end{enumerate}
        Then, it holds that $\lim_{k \to \infty} \E[\|\Phi(x^k)\|] = 0$. 
		\item Let the properties \ref{P1}--\ref{P2} hold and assume further that 
		\begin{enumerate}[label=\textup{\textrm{(P.\arabic*${}^\prime$)}},topsep=0pt,itemsep=0ex,partopsep=0ex,start=3]
			\item \label{P3'} there exist constants ${A}, b \geq 0$, $p_1, p_2, q > 0$ and random vectors $\bs{A}_k, \bs{B}_k : \Omega \to \R^n$ such that
			\[ x^{k+1} = x^k + \mu_k^{p_1} \bs{A}_k + \mu_k^{p_2} \bs{B}_k, \]
			and for all $k$,  $\bs{A}_k, \bs{B}_k$ are $\mathcal F_{k+1}$-measurable, and we have $\E[\bs{A}_k \mid \mathcal F_k] = 0$ almost surely, $\E[\|\bs{A}_k\|^q] \leq {A}$, and $\limsup_{k \to \infty} \|\bs{B}_k\|^q /(1+\|\Phi(x^k)\|^b) < \infty$ almost surely,
			\item \label{P4'} the sequence $\{\mu_k\}_{k\geq 0}$  and the  parameters $a, b, q, p_1, p_2$ satisfy $\mu_k \to 0$,
			\begin{align*}
                &{\sum}_{k=0}^{\infty} \ \mu_k = \infty, \quad {\sum}_{k=0}^\infty \, \mu_k^{2p_1} < \infty,\\
                &\text{and} \quad q \geq 2, \;\; qa \geq b, \ \ p_1 > \frac{1}{2}, \;\; p_2 \geq 1.
            \end{align*}
		\end{enumerate}
		Then, it holds that $\lim_{k \to \infty} \|\Phi(x^k)\| = 0$ almost surely. 
	\end{enumerate}
	\end{theorem}

	We next verify the conditions in Theorem \ref{thm:unified} by setting $\Phi:= \nabla f$. Then \ref{P1} is satisfied for $L_{\Phi} = L$. According to Lemma \ref{lem:combined} and the supermartingale convergence theorem, we have $\sum_{t=0}^{\infty}\alpha_t \E[\norm{\nabla f(x_t)}^2]<\infty$. Hence, condition \ref{P2} is satisfied with $a = 2$ and $\mu_t = \alpha_t$. In addition, we have the sequence $\crk{\E[f(x_t) - \uf]}$ converges to some finite value.
	From \eqref{eq:xt}, we have 
	\begin{equation}
        \label{eq:AB}
        \begin{aligned}
            &x_{t+1}
            = x_t + \alpha_t \underbrace{\frac{1}{n}\sumn \sum_{\ell=0}^{Q-1}\nabla f_i(\xitl)}_{\bs{B}_t}\\
            &\quad- \alpha_t \underbrace{\brk{\frac{1}{n}\sumn\sum_{\ell = 0}^{Q-1} g_i(\xitl;\xi_{i,t}^\ell) - \frac{1}{n}\sumn \sum_{\ell=0}^{Q-1}\nabla f_i(\xitl)}}_{\bs{A}_t}.
        \end{aligned}
    \end{equation}

	\begin{enumerate}
		\item Verifying \ref{P3} and \ref{P4}. On one side, substituting \eqref{eq:local_x} into \eqref{eq:local_fi} 
            leads to 
		\begin{equation}
            \label{eq:fi_finite}
            \begin{aligned}
                &\sum_{\ell=0}^{Q-1}\sumn \E\brk{f_i(\xitl) - \uf_i}\leq \mathcal{C}_1 + \cC_2\E\brk{f(x_t) - \uf},\\
                &\quad \exists \text{ some constant }\mathcal{C}_1, \cC_2 \geq 0.
            \end{aligned}
        \end{equation}

            On the other side, \eqref{eq:local_x} and \eqref{eq:local_x_s2} imply that
		\begin{equation}
            \label{eq:xitl_finite}
            \begin{aligned}
                \sum_{\ell=0}^{Q-1}\sumn \E\brk{\norm{\nabla f_i(\xitl)}^2} & \leq \cC_3 + \cC_4\E\brk{\norm{\nabla f(x_t)}^2}\\
                &\quad + \cC_5 \E\brk{f(x_t) - \uf},
            \end{aligned}
        \end{equation}
		for some positive constants $\cC_3, \cC_4$, and $\cC_5$. Therefore, we have from \eqref{eq:AB} that 
		\begin{equation}
			\label{eq:verify_P3}
			\begin{aligned}
				&\E\brk{\norm{x_{t + 1} - x_t}^2} \leq \frac{\alpha_t^2QC}{n}\sum_{\ell=0}^{Q-1}\sumn \E\brk{f_i(\xitl) - \uf_i}\\
                &\quad + \alpha_t^2Q^2 D + \frac{\alpha_t^2 Q}{n}\sum_{\ell=0}^{Q-1}\sumn \E\brk{\norm{\nabla f_i(\xitl)}^2}.
			\end{aligned}
		\end{equation}

		Combining \eqref{eq:fi_finite}-\eqref{eq:verify_P3} verifies \ref{P3} and \ref{P4} for $p_1 = p_2 = b = q= 2$.
		\item Verifying \ref{P3'} and \ref{P4'}. It can be verified from the tower property that $\condEi{\bs{A}_t}{\cF_t} =0$ almost surely. In addition, Assumption \ref{as:abc} and \eqref{eq:fi_finite} verifies that $\E[\normi{\bs{A}_t^2}]$ is bounded. We next verify $\limsup_{k \to \infty} \|\bs{B}_k\|^q /(1+\|\Phi(x^k)\|^b) < \infty$ almost surely, which can be seen from \eqref{eq:xitl_finite}. 
	\end{enumerate}

	We conclude that $\lim_{t\rightarrow\infty}\E[\normi{\nabla f(x_t)}] = 0$ and $\lim_{t\rightarrow}\normi{\nabla f(x_t)} = 0$ almost surely.


\subsection{Verifying the General Variance Condition}
\label{app:verify}

    As shown in \cite[Section 9.1]{khaled2020better}, the stochastic gradient $g_1(x;\xi)$ of $f_1(x)$ violates the relaxed growth condition \eqref{eq:intro_bv1}, but satisfies Assumption \ref{as:abc}. We next show that this argument also applies to the stochastic gradient $g_i(x;\xi)$ of $f_i(x)$, where $i=2,3,\cdots,16$. We have 
    \begin{equation}
        \label{eq:g2sq}
        \begin{aligned}
            \E_{\xi}\brk{\norm{g_i(x;\xi)}^2} = \prt{\frac{1}{1 + e^{-(x-i+1)}}}^2 + \ln(1 + e^{x-i + 1}).
        \end{aligned}
    \end{equation}
    To show the relaxed growth condition \eqref{eq:intro_bv1} does not hold, it suffices to show that for any $\sigma,\eta\geq 0$ but not equal to zero simultaneously, there exists $x\in\R$ such that 
    \begin{equation}
        \label{eq:rg_not}
        \E_{\xi}\brk{\norm{g_i(x;\xi)}^2}>\sigma^2 + (\eta^2 + 1)\norm{\nabla f_i(x)}^2.
    \end{equation}

    We define $h_i(x)$ for $i=2,3,\cdots, 16$ as 
    \begin{equation}
        \label{eq:hx}
        h_i(x):= \ln(1 + e^{x - i + 1}) - \eta^2\prt{\frac{1}{1+e^{-(x-i+1)}}}^2-\sigma^2.
    \end{equation}

    Noting that $1/(1+e^{-x})\in (0, 1),\forall x\in\R$, we have $h_i(x)> \ln(1 + e^{x-i+1}) - \eta^2 -\sigma^2.$
    Thus, we obtain $h_i(x)>0$ for all $x\geq [\ln(e^{(\eta^2 + \sigma^2)}-1)+i - 1]$, implying that the relaxed growth condition \eqref{eq:intro_bv1} is not satisfied for these values of $x$. This is true for each $i=2,3,\cdots, 16$.

    We next show that for any $i=2,3,\cdots, 16$, the stochastic gradient $g_i(x;\xi)$ satisfies the ABC assumption \eqref{eq:abc}. Noting that $f_i^*=0$, we have 
    \begin{align*}
        &\E_{\xi}\brk{\norm{g_i(x;\xi) - \nabla f_i(x)}^2}= \E_{\xi}\brk{\norm{g_i(x;\xi)}^2} - \norm{\nabla f_i(x)}^2\\
        &\leq \prt{\frac{1}{1 + e^{-(x-i+1)}}}^2 + \ln(1 + e^{x-i + 1})\\
        &\leq 1 + f_i(x) - f_i^*, i=2,3,\cdots, 16.
    \end{align*}
    Hence, the stochastic gradient $g_i(x;\xi)$ satisfies the ABC assumption \eqref{eq:abc}. We also note that the above construction can be applied to any function whose gradient is bounded and whose function value can tend towards infinity.

{
\section{Proofs for SCAFFOLD}
\label{app:scaffold}

\subsection{Supporting Lemmas for SCAFFOLD}

The analysis for SCAFFOLD generally aligns with the procedures in \cite[Lemma 16-Lemma 19]{karimireddy2020scaffold}, with some notable adaptations. Specifically, we replace the bounded variance assumption for stochastic gradients with Assumption \ref{as:abc} and assume the full participation of all the agents for simplification.
Furthermore, we introduce a Lyapunov function to address the challenge because of considering Assumption \ref{as:abc}.

Referring to Algorithm \ref{alg:scaffold}, we can express the SCAFFOLD \cite{karimireddy2020scaffold} update as follows. 

Recall that we denote $g_{i,t}^s:= g_i(x_{i,t}^s;\xi_{i,t}^s)$. The agent updates as in \eqref{eq:scaffold_agent}:
\begin{subequations}
    \label{eq:scaffold_agent}
    \begin{align}
        x_{i,t}^{\ell + 1} &= 
            x_{i,t}^\ell - \eta_a \prt{g_{i,t}^\ell - c_{i,t} + c_t},\label{eq:scaffold_xi}\\
        c_{i,t + 1} &= c_{i,t} - c_t + \frac{1}{Q\eta_a}\prt{x_t - x_{i,t}^Q}.\label{eq:scaffold_ci}
    \end{align}
\end{subequations}

The server updates as in \eqref{eq:scaffold_server}:
\begin{subequations}
    \label{eq:scaffold_server}
    \begin{align}
        x_{t + 1} &= x_t + \frac{\eta_s}{n} \sum_{i=1}^n\prt{x_{i,t}^Q - x_t},\label{eq:scaffold_x}\\
    c_{t + 1} &= c_t + \frac{1}{n}\sum_{i=1}^n\prt{c_{i,t+1} - c_{i,t}}.\label{eq:scaffold_c}
    \end{align}
\end{subequations}

Therefore, we have for $\ell = 0,1,\ldots, Q-1$ that 
\begin{equation}
    \label{eq:scaffold_x_diff}
    x_{i,t}^{\ell + 1} = x_{i,t}^0 - \eta_a\sum_{s=0}^\ell g_{i,t}^s - \eta_a(\ell + 1)\prt{c_t - c_{i,t}}.
\end{equation}

Letting $\ell = Q-1$ in \eqref{eq:scaffold_x_diff} and noting $x_{i,t}^0 = x_t$, we obtain
\begin{equation}
    \label{eq:scaffold_xq}
    x_{i,t}^Q - x_t = - \eta_a\sum_{s=0}^{Q-1} g_{i,t}^s - \eta_a Q\prt{c_t - c_{i,t}}.
\end{equation}

Substituting \eqref{eq:scaffold_xq} into \eqref{eq:scaffold_ci} and \eqref{eq:scaffold_x} yields 

\begin{align}
    c_{i,t + 1} &= \frac{1}{Q}\sum_{s=0}^{Q-1} g_{i,t}^s, \label{eq:scaffold_ci_sum}\\
    x_{t + 1} &= x_t - \frac{\teta}{nQ}\sum_{i=1}^n\sum_{s=0}^{Q-1}\brk{g_{i,t}^s + \prt{c_t - c_{i,t}}}:= x_t + \Delta x_t, \label{eq:scaffold_xt_sum}
\end{align}
where $\teta= \eta_a\eta_s Q$ denotes the effective stepsize.
Similar to those in Appendix \ref{app:lem_descent}, it is critical to estimate $\E\brki{\normi{\Delta x_t}^2}$. 
Note that $\sigfn= f^* - \sumn\uf_i/n$ in the following. Similar to those in \cite[Appendix E.2]{karimireddy2020scaffold}, we define 
    \begin{equation}
        \label{eq:scaffold_add}
        \begin{aligned}
            \Xi_t&:= \frac{1}{nQ}\sumn\sum_{s = 0}^{Q-1}\E\brk{\norm{x_{i,t-1}^s - x_{t}}^2},\ t = 1, 2,\ldots, T,\\
            \mathcal{E}_t &:= \frac{1}{nQ}\sumn\sum_{s = 0}^{Q-1}\E\brk{\norm{x_{i,t}^s - x_t}^2},\ t = 0,1,\ldots, T - 1.
        \end{aligned}
    \end{equation}

\begin{lemma}
    \label{lem:dxt}
    Let Assumptions \ref{as:abc} and \ref{as:smooth} hold. Then,
    \begin{equation}
        \label{eq:dxt}
        \begin{aligned}
            &\E\brk{\norm{\Delta x_t}^2} \leq 7\teta^2 L(L+C) \cE_t + 28\teta^2 L^2\E\brk{\norm{\Delta x_{t -1}}^2} \\
            &\quad + 14\teta^2 L(C + 2L)\cE_{t-1} + 7\teta^2\E\brk{\norm{\nabla f(x_t)}^2}\\
            &\quad + \frac{14\teta^2C}{nQ}\E\brk{f(x_t) - f^* } + \frac{21\teta^2 (2C\sigfn+D)}{nQ}\\
            &\quad + \frac{28\teta^2C}{nQ}\E\brk{f(x_{t-1}) - f^*}.
        \end{aligned}
    \end{equation}
\end{lemma}

\begin{proof}
    According to \eqref{eq:scaffold_ct}, we can split $\Delta x_t$ as follows.
    \begin{equation}
        \label{eq:dxt_split}
        \begin{aligned}
            &\Delta x_t= -\frac{\teta}{n Q}\sum_{i=1}^n\sum_{s=0}^{Q-1}\crk{ \brk{g_{i,t}^s - \nabla f_i(x_{i,t}^s)}\right.\\
            &\left. \quad + \brk{\nabla f_i(x_{i,t}^s) - \nabla f_i(x_{t})} + \brk{\nabla f_i(x_t) - \nabla f_i(x_{i, t - 1}^s)}\right.\\
            &\left.\quad + \brk{\nabla f_i(x_{i,t - 1}^s) - c_{i,t}} + \brk{c_t - \nabla f(x_{i,t - 1}^s)}\right.\\
            &\left.\quad + \brk{\nabla f(x_{i,t - 1}^s) - \nabla f(x_t)} + \nabla f(x_t)}.
        \end{aligned}
    \end{equation}


    It is worth noting that 
    \begin{equation}
        \label{eq:Xi_cE}
        \begin{aligned}
            &\Xi_{t + 1} 
            \leq 2\cE_t + 2\E\brk{\norm{\Delta x_t}^2}.
        \end{aligned}
    \end{equation}
    Therefore, we next estimate $\E\brki{\normi{\Delta x_t}^2}$ by $\mathcal{E}_t$ in light of \eqref{eq:dxt_split} and Assumption \ref{as:abc}. 

    Similar to those in \eqref{eq:var1_s1}-\eqref{eq:var_t}, we have for the first difference term in \eqref{eq:dxt_split} that 
    \begin{equation}
        \label{eq:dxt_var_t}
        \begin{aligned}
            &\condE{\norm{\frac{\teta}{n Q}\sum_{i=1}^n\sum_{s=0}^{Q-1} \brk{g_{i,t}^s - \nabla f_i(x_{i,t}^s)}}^2}{\cF_t}\\
            &\leq \frac{\teta^2C}{n^2Q^2}\sumn\sum_{s=0}^{Q-1}\condE{f_i(x_{i,t}^s) - \uf_i}{\cF_t} + \frac{\teta^2 D}{nQ}.
        \end{aligned}
    \end{equation}

    In light of the $L$-smoothness of $f_i$ for any $i\in[n]$ in Assumption \ref{as:smooth}, we have 
    
    \begin{equation}
        \label{eq:dxt_smooth}
        \begin{aligned}
            &{\norm{\frac{\teta}{n Q}\sum_{i=1}^n\sum_{s=0}^{Q-1}\brk{\nabla f_i(x_{i,t}^s) - \nabla f_i(x_{t})}}^2}\\
            &\leq \frac{\teta^2 L^2}{nQ}\sumn\sum_{s=0}^{Q-1}{\norm{x_{i,t}^s - x_t}^2},\\
            &{\norm{\frac{\teta}{n Q}\sum_{i=1}^n\sum_{s=0}^{Q-1}\brk{\nabla f_i(x_t) - \nabla f_i(x_{i,t-1}^s)}}^2}\\
            &\leq \frac{\teta^2 L^2}{nQ}\sumn\sum_{s=0}^{Q-1}{\norm{x_{t} - x_{i,t-1}^s}^2},\\
            &{\norm{\frac{\teta}{n Q}\sum_{i=1}^n\sum_{s=0}^{Q-1}\brk{\nabla f(x_{i,t-1}^s) - \nabla f(x_{t})}}^2}\\
            &\leq \frac{\teta^2 L^2}{nQ}\sumn\sum_{s=0}^{Q-1}{\norm{x_{i,t-1}^s - x_t}^2}.
        \end{aligned}
    \end{equation}

    Noting the relation for $c_{i,t}$ in \eqref{eq:scaffold_ci_sum}, we have 
    \small
    \begin{equation}
        \label{eq:dxt_cit}
        \begin{aligned}
            &\condE{\norm{\frac{\teta}{n Q}\sum_{i=1}^n\sum_{s=0}^{Q-1}\brk{\nabla f_i(x_{i,t-1}^s) - c_{i,t}}}^2}{\cF_{t-1}}\\
            &= \condE{\norm{\frac{\teta}{n Q}\sum_{i=1}^n\sum_{s=0}^{Q-1}\brk{\nabla f_i(x_{i,t-1}^s) - \frac{1}{Q}\sum_{r = 0}^{Q-1} g_{i,t-1}^r }}^2}{\cF_{t-1}}\\
            &\leq \frac{\teta^2C}{n^2Q^2}\sumn\sum_{s=0}^{Q-1}\condE{f_i(x_{i,t-1}^s) - \uf_i}{\cF_{t-1}} + \frac{\teta^2 D}{nQ},
        \end{aligned}
    \end{equation}\normalsize
    where the last inequality holds by similar derivation as in \eqref{eq:dxt_var_t}.

   Similarly, we have from $c_t = \sumn c_{i,t}/n$ in \eqref{eq:scaffold_ct} that 
    \begin{equation}
        \label{eq:dxt_ct}
        \begin{aligned}
            &\condE{\norm{\frac{\teta}{n Q}\sum_{i=1}^n\sum_{s=0}^{Q-1}\brk{\nabla f_i(x_{i,t-1}^s) - c_{t}}}^2}{\cF_{t-1}}\\
            &\leq \frac{\teta^2C}{n^2Q^2}\sumn\sum_{s=0}^{Q-1}\condE{f_i(x_{i,t-1}^s) - \uf_i}{\cF_{t-1}} + \frac{\teta^2 D}{nQ}.
        \end{aligned}
    \end{equation}

    Combing \eqref{eq:dxt_var_t}-\eqref{eq:dxt_ct} and noting that $\cF_{t-1}\subset \cF_t$ lead to 
    \small
    \begin{equation}
        \label{eq:dxt_cond}
        \begin{aligned}
            &\frac{1}{7}\condE{\norm{\Delta x_t}^2}{\cF_{t-1}} \leq \frac{\teta^2 L^2}{nQ}\sumn\sum_{s=0}^{Q-1} \condE{\norm{x_{i,t}^s - x_t}^2}{\cF_{t-1}}\\
            &\quad +  \frac{\teta^2C}{n^2Q^2}\sumn\sum_{s=0}^{Q-1}\condE{f_i(x_{i,t}^s) - \uf_i}{\cF_{t-1}} + \frac{3\teta^2 D}{nQ}\\
            &\quad + \frac{2\teta^2C}{n^2Q^2}\sumn\sum_{s=0}^{Q-1}\condE{f_i(x_{i,t-1}^s) - \uf_i}{\cF_{t-1}}\\
            &\quad + \frac{2\teta^2 L^2}{nQ}\sumn\sum_{s=0}^{Q-1}\condE{\norm{x_{t} - x_{i,t-1}^s}^2}{\cF_{t-1}}\\
            &\quad + \teta^2\condE{\norm{\nabla f(x_t)}^2}{\cF_{t-1}}.
        \end{aligned}
    \end{equation}\normalsize

    It is worth noting that Lemma \ref{lem:local_fi} relies solely on the smoothness of $f_i$ for any $i\in[n]$. Consequently, we can employ Lemma \ref{lem:local_fi} in the context of \eqref{eq:dxt_cond}. After that, taking the full expectation and invoking the definitions of $\Xi_t$ and $\cE_t$ in \eqref{eq:scaffold_add} yield the intended result \eqref{eq:dxt}.

\end{proof}




\begin{lemma}
    \label{lem:cEt}
    Let Assumptions \ref{as:abc} and \ref{as:smooth} hold. Set the stepsize $\eta_a$ to satisfy $\eta_a\leq \sqrt{1/(7Q^2L(L+2C))}$. Then,
    \begin{equation}
        \label{eq:cEt}
        \begin{aligned}
            &\cE_{t + 1} \leq 28\eta_a^2 Q C \E\brk{f(x_{t + 1}) - f^*} + 42\eta_a^2Q(2C \sigfn+D)\\
            &\quad + 56\eta_a^2 Q C \E\brk{f(x_{t}) - f^*} + 7\eta_a^2Q^2L (4C+3L)\cE_{t}\\
            &\quad + 7\eta_a^2 Q^2\E\brk{\norm{\nabla f(x_t)}^2}  + 14Q^2\eta_a^2 L^2\E\brk{\norm{\Delta x_t}^2}.
        \end{aligned}
    \end{equation}
\end{lemma}

\begin{proof}
    In light of \eqref{eq:scaffold_x_diff}, we have 
    \begin{equation}
        \label{eq:Xit_split} 
        \begin{aligned}
            & x_{i,t + 1}^s - x_{t + 1} = -\eta_a\sum_{r = 0}^{s-1}g_{i,t+1}^r -\eta_a s\prt{c_{t + 1} - c_{i,t+1}}\\
            &= -\eta_a\sum_{r = 0}^{s-1}\crk{\brk{g_{i,t+1}^r - \nabla f_i(x_{i,t+1}^r)} \right.\\
            &\left.\quad + \brk{\nabla f_i(x_{i,t+1}^r) - \nabla f_i(x_{t + 1})}} -\eta_a s \nabla f(x_t) \\
            &\quad - \frac{\eta_a s}{Q}\sum_{r=0}^{Q-1}\crk{ \brk{\nabla f_i(x_{t + 1}) - \nabla f_i(x_{i,t}^r)}\right.\\
            &\left.\quad + \brk{\nabla f_i(x_{i,t}^r) - c_{i,t+1}}}\\
            &\quad - \frac{\eta_a s}{nQ}\sumn\sum_{r=0}^{Q-1}\crk{\brk{c_{t+1} - \nabla f(x_{i,t}^r)}\right.\\
            &\left.\quad + \brk{\nabla f(x_{i,t}^r) - \nabla f(x_t)}} .
        \end{aligned}
    \end{equation}

    Notably, relation \eqref{eq:Xit_split} is similar to \eqref{eq:dxt_split}. Following the derivations in the proof of Lemma \ref{lem:dxt}, we obtain

    \begin{equation}
        \label{eq:Xit}
        \begin{aligned}
            &\frac{1}{7\eta_a^2}\cE_{t + 1} 
            \leq2 Q C \E\brk{f(x_{t + 1}) - f^* } + \frac{Q^2L (2C+L)}{2}\cE_{t + 1}\\
            &\quad + 4 Q C \E\brk{f(x_{t}) - f^* } + \frac{Q^2L (4C+L)}{2}\cE_{t}\\
            &\quad + 3 Q(D+2C\sigfn) + \frac{ Q^2}{2}\E\brk{\norm{\nabla f(x_t)}^2}  + \frac{Q^2 L^2}{2}\Xi_{t + 1}.
        \end{aligned}
    \end{equation}

    Letting $\eta_a\leq \sqrt{1/7Q^2L(L+2C)}$ and invoking \eqref{eq:Xi_cE} yield the desired result.
\end{proof}

It is worth noting that Lemma \ref{lem:descent} relies on relation \eqref{eq:xt}, which also holds for SCAFFOLD due to \eqref{eq:scaffold_ct}, except that we use stepsize $\eta_s\eta_a$ here: 
\begin{align*}
    x_{t + 1} 
    &= x_t - \frac{\eta_s\eta_a}{n}\sum_{i=1}^n\sum_{s=0}^{Q-1}g_i(x_{i,t}^s;\xi_{i,t}^s).
\end{align*}
Such an observation leads to the following lemma. 

\begin{lemma}
    \label{lem:scaffold_descent}
    Let Assumptions \ref{as:abc} and \ref{as:smooth} hold. Set the effective stepsize $\teta$ to satisfy $\teta\leq \min\crki{1/L, 1/C}$. Then,
    \begin{equation}
        \label{eq:scaffold_descent}
        \begin{aligned}
            &\E\brk{f(x_{t + 1}) - f^*} \leq \prt{1 + \frac{\teta^2 C L}{nQ}} \E\brk{f(x_t) - f^*}\\
            &\quad - \frac{\teta}{2}\E\brk{\norm{\nabla f(x_t)}^2} + \teta L^2\cE_t + \frac{\teta^2 L (D + 2C \sigfn)}{2nQ}.
        \end{aligned}
    \end{equation}
\end{lemma}

In light of Lemmas \ref{lem:dxt}-\ref{lem:scaffold_descent}, we consider the Lyapunov function $\cL_k$: 
\begin{equation}
    \label{eq:scaffold_lya_can}
    \begin{aligned}
        \cL_k:= \E\brk{f(x_t) - f^*} + \teta \cB_1\cE_t +  \teta \cB_2\E\brk{\norm{\Delta x_t}^2},
    \end{aligned}
\end{equation}
where the positive constants $\cB_1,\cB_2$ are determined later.
\begin{lemma}
    \label{lem:scaffold_lya}
    Let Assumptions \ref{as:abc} and \ref{as:smooth} hold. Set the effective stepsize $\teta$ to satisfy 
    \begin{align*}
        \teta\leq \min\crk{\sqrt{\frac{nQ}{CL}}, \frac{1}{12(L+C)}, \frac{\eta_s}{\sqrt{84 L(L+C)}} }.
    \end{align*}
    Then,
    \begin{equation}
        \label{eq:scaffold_lya}
        \begin{aligned}
            &\cL_{k + 1} \leq \crk{1 + \frac{\teta^2 CL}{nQ} + \frac{280\teta^3 LC(L+C)}{\eta_s^2 Q}}\cL_k\\
            &\quad - \frac{\teta}{4}\E\brk{\norm{\nabla f(x_t)}^2} + \frac{\teta^2 L (2C\sigfn + D)}{2nQ}\\
            &\quad + \frac{224 \teta^3 L^2(2C\sigfn + D)}{Q\eta_s^2}.\\
        \end{aligned}
    \end{equation}
\end{lemma}

\begin{proof}
    Substituting \eqref{eq:scaffold_descent} into \eqref{eq:cEt} leads to 
    \begin{equation}
        \label{eq:cEt_new}
        \begin{aligned}
            &\cE_{t + 1} \leq 112\eta_a^2QC \E\brk{f(x_t) - f^*} + 28\eta_a^2Q^2L(L+C)\cE_t\\
            &\quad + 14Q^2\eta_a^2 L^2\E\brk{\norm{\Delta x_t}^2}+ 7\eta_a^2Q^2\E\brk{\norm{\nabla f(x_t)}^2}\\
            &\quad + 56\eta_a^2Q(2C\sigfn + D),
        \end{aligned}
    \end{equation}
    where we let the effective stepsize $\teta$ satisfy $\teta\leq \min\crki{\sqrt{nQ/(CL)}, Q/(4C)}$.

    We also have 
    \begin{align}
        \norm{\nabla f(x_{t + 1})}^2\leq 2L^2\norm{\Delta x_t}^2 + 2\norm{\nabla f(x_t)}^2\label{eq:nf_diff}.
    \end{align}

    Substituting \eqref{eq:scaffold_descent}, \eqref{eq:cEt_new}, and \eqref{eq:nf_diff} into \eqref{eq:dxt} yields 
    \begin{equation}
        \label{eq:dxt_new}
        \begin{aligned}
            &\E\brk{\norm{\Delta x_{t + 1}}^2} 
            \leq 70\teta^2 L^2\E\brk{\norm{\Delta x_t}^2} + 56\teta^2 L(L+C)\cE_t\\
            &\quad + \frac{168\teta^2 C}{Q}\E\brk{f(x_t) - f^*} + 21\teta^2\E\brk{\norm{\nabla f(x_t)}^2}\\
            &\quad +  \frac{28\teta^2(2C\sigfn + D)}{nQ}  + \frac{392\teta^4 L(L+C)(2C\sigfn + D)}{Q\eta_s^2},
        \end{aligned}
    \end{equation}
    where we let $\eta_a\leq \sqrt{1/(7Q^2 L(L+C))}$ and $\teta\leq 1/C$.

    Substituting \eqref{eq:scaffold_descent}, \eqref{eq:cEt_new}, and \eqref{eq:dxt_new} into \eqref{eq:scaffold_lya_can} yields 
    \begin{equation}
        \label{eq:scaffold_lya_s1}
        \begin{aligned}
            &\cL_{k + 1}\leq a\E\brk{f(x_t) - f^*} +  \frac{\teta^2 L (2C\sigfn + D)}{2nQ}\\
            &\quad + \brk{\teta L^2+ \frac{28\teta^3 L(L+C)\cB_1}{\eta_s^2} + 56\teta^3 L(L+C)\cB_2}\cE_t\\
            &\quad + \prt{\frac{14\teta^3 L\cB_1}{\eta_s^2} + 70\teta^3L^2\cB_2 }\E\brk{\norm{\Delta x_t}^2}\\
            &\quad - \frac{\teta}{2}\prt{1 - \frac{14\teta^2\cB_1}{\eta_s^2} - 42\teta^2\cB_2}\E\brk{\norm{\nabla f(x_t)}^2}\\
            &\quad + \frac{56 \teta^3 \cB_1(2C\sigfn + D)}{Q\eta_s^2} + \frac{28\teta^3\cB_2 (2C\sigfn + D)}{nQ}\\
            &\quad + \frac{392\teta^5  L(L+C)\cB_2(2C\sigfn + D)}{Q\eta_s^2},
        \end{aligned}
    \end{equation}
    where 
    \begin{align*}
        a&:= 1 + \frac{\teta^2 CL}{nQ} + \frac{112\teta^3 C\cB_1}{Q\eta_s^2} + \frac{168\teta^3 C \cB_2}{Q}.
    \end{align*}

    To derive the recursion for $\cL_k$, it suffices to have 
    \small
    \begin{subequations}
        \label{eq:cBs_can}
        \begin{align}
            & L^2 + 56\teta^2 L(L+C)\cB_2 \leq \prt{1 - \frac{28\teta^2 L(L+C)}{\eta_s^2}}\cB_1,\label{eq:cB1_can}\\
            &\frac{14\teta^2 L\cB_1}{\eta_s^2} \leq \prt{1 - 70\teta^2 L^2}\cB_2. \label{eq:cB2_can}
        \end{align}
    \end{subequations}\normalsize

    Letting $\teta \leq \sqrt{1/(140 L^2)}$, we can choose $\cB_2 = 28\teta^2 L^2\cB_1/\eta_s^2$ according to \eqref{eq:cB2_can}. Then, the relation \eqref{eq:cB1_can} becomes 
    \begin{align*}
       L^2\leq \prt{1 - \frac{28\teta^2 L(L+C)(1 + 56\teta^2 L^2)}{\eta_s^2}}\cB_1.
    \end{align*}

    Thus, we can choose $\cB_1 = 2L^2$ by letting $\teta\leq \eta_s/\sqrt{84L(L+C)}$. Then, substituting $\cB_1$ and $\cB_2$ into \eqref{eq:scaffold_lya_s1} leads to the desired result \eqref{eq:scaffold_lya}.

\end{proof}

\subsection{Proof of Theorem \ref{thm:scaffold}}
\label{app:thm_scaffold}

\begin{proof}
    Applying Lemma \ref{lem:min} to \eqref{eq:scaffold_lya} leads to 
    \begin{equation}
        \label{eq:scaffold_minE_s1}
        \begin{aligned}
            &\min_{t=0,1,\ldots, T-1} \E\brk{\norm{\nabla f(x_t)}^2}\leq \frac{12 \cL_0}{\teta T} + \frac{2\teta L (2C\sigfn + D)}{nQ}\\
            &\quad + \frac{896\teta^2 L(L+C)(2C\sigfn + D)}{Q\eta_s^2},
        \end{aligned}
    \end{equation}
    where we invoke the condition of $\teta$ in \eqref{eq:scaffold_teta}.
    
    
    From \eqref{eq:scaffold_x_diff}, we have
    \begin{align*}
        &x_{i,0}^{\ell} - x_0 = - \eta_a\ell\prt{c_0 - c_{i,0} + \nabla f_i(x_0)}\\
        &\quad - \eta_a\sum_{s=0}^{\ell-1} \crk{\brk{g_{i,0}^s - \nabla f_i(x_{i,0}^s)} + \brk{\nabla f_i(x_{i,0}^s) - \nabla f_i(x_0)} }.
    \end{align*} 
    
    Denote $\Delta_0:= f(x_0) - f^*$, $\Delta_1:= \sumn\norm{c_0 - c_{i,0}}^2/n$, and $\Delta_2:= \sumn\norm{\nabla f_i(x_0)}^2 / n$. Similar to \eqref{eq:Xit_split}-\eqref{eq:Xit}, we have 
    
    \begin{equation}
        \label{eq:cE0}
        \begin{aligned}
            & \frac{1}{4\eta_a^2}\cE_0 \leq 4 QC\Delta_0 +  Q^2\Delta_1 + Q^2\Delta_2+ 2 Q(2C\sigfn + D).
        \end{aligned}
    \end{equation}
    
    From \eqref{eq:scaffold_ct} and \eqref{eq:dxt_split}, we have 
    \begin{align*}
        &\Delta x_0= -\frac{\teta}{n Q}\sum_{i=1}^n\sum_{s=0}^{Q-1}\crk{ \brk{g_{i,0}^s - \nabla f_i(x_{i,0}^s)}\right.\\
        &\left. \quad + \brk{\nabla f_i(x_{i,0}^s) - \nabla f_i(x_{0})} + \nabla f_i(x_0)}.
    \end{align*}
    
    Similar to those in the proof of Lemma \ref{lem:dxt}, we have 
    \begin{equation}
        \label{eq:dxt0}
        \begin{aligned}
            &\E\brk{\norm{\Delta x_0}^2} 
            \leq \frac{15\teta^2 C\Delta_0}{Q}  + \frac{12\teta^4 L(L+C)}{\eta_s^2}\Delta_1 + 6\teta^2\Delta_2\\
            &\quad + \frac{15\teta^2(2C\sigfn + D)}{4Q},
        \end{aligned}
    \end{equation}
    where we invoke $\teta\leq \eta_s/\sqrt{84 L(L+C)}$. Substituting \eqref{eq:cE0} and \eqref{eq:dxt0} into \eqref{eq:scaffold_minE_s1} leads to the desired result \eqref{eq:scaffold_minE}.
    

    If we set $\teta$ as in \eqref{eq:teta_etas}, then, 
    \begin{equation}
        \label{eq:teta_etas_upper}
        \begin{aligned}
            \frac{1}{\teta T} &= \sqrt{\frac{L^2 + C^2}{2nQ T}} + \frac{\sqrt{84 L(L+C)}}{\sqrt{\eta_s} T} + \frac{12 (L+C)}{T}.
        \end{aligned}
    \end{equation}

    Substituting \eqref{eq:teta_etas_upper} into \eqref{eq:scaffold_minE} yields the desired result \eqref{eq:scaffold_minE_etas}.

\end{proof}

}

\bibliographystyle{IEEEtran}
\bibliography{references_all}

\begin{IEEEbiography}[{\includegraphics[width=1in,height=1.25in,clip,keepaspectratio]{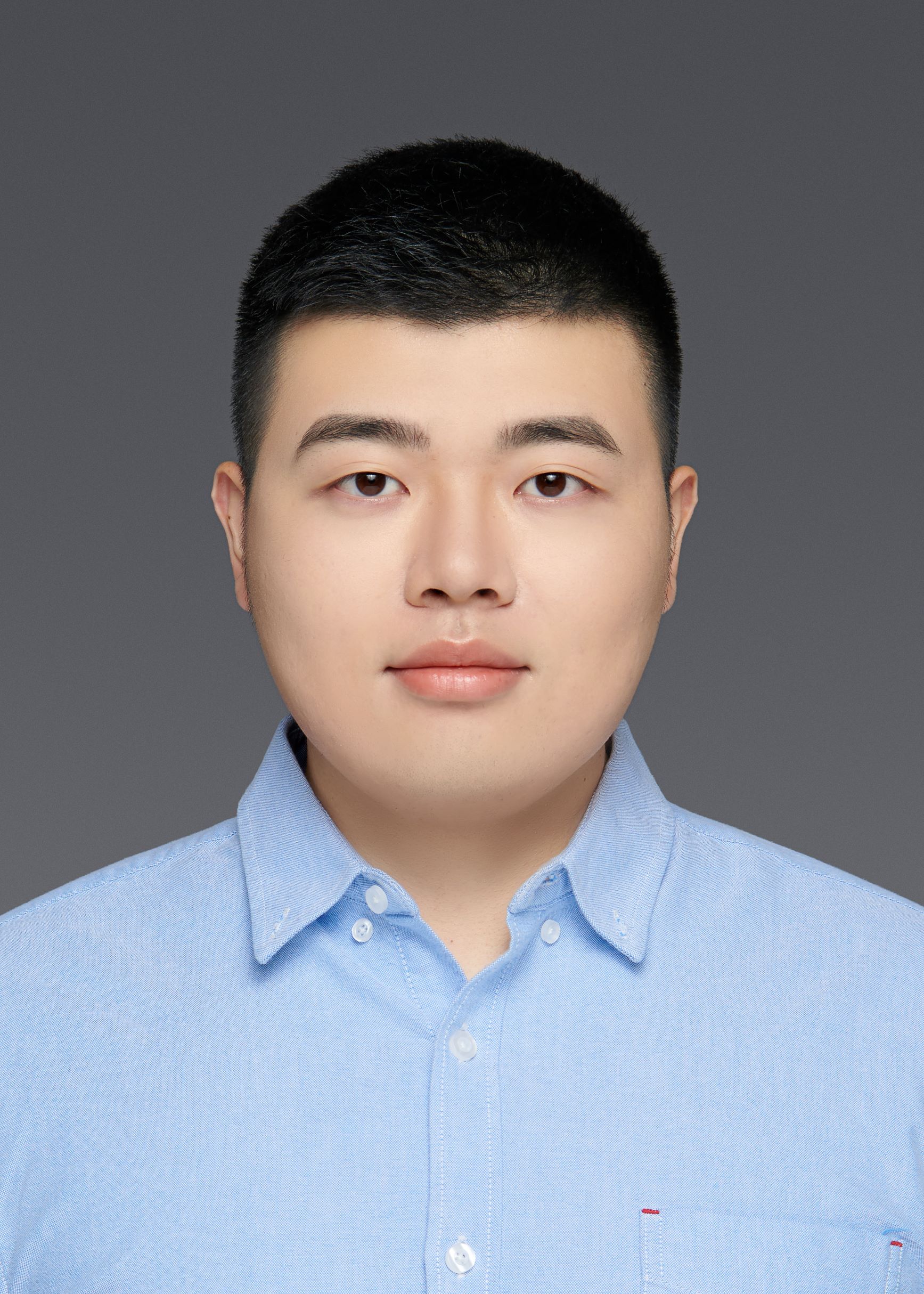}}]
    {Kun Huang} is currently a Ph.D. candidate in data science at the School of Data Science, The Chinese University of Hong Kong, Shenzhen, China. He obtained a B.S. degree in Applied Mathematics from Tongji University in 2018, and an M.S. degree in Statistics from the University of Connecticut in 2020. His research interests primarily lie in the fields of distributed optimization and machine learning.
    \end{IEEEbiography}
    
    \begin{IEEEbiography}[{\includegraphics[width=1in,height=1.25in,clip,keepaspectratio]{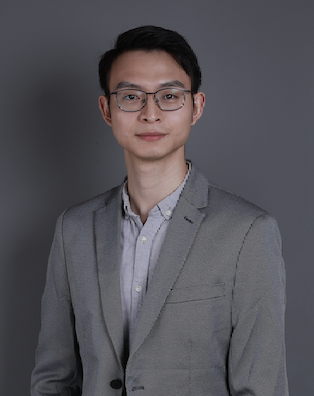}}]
      {Xiao Li} is an assistant professor at the School of Data Science at the Chinese University of Hong Kong, Shenzhen. He received his Ph.D. degree in electronic engineering from the Chinese University of Hong Kong in 2020 and his B.Eng. degree in communication engineering from Zhejiang University of Technology in 2016. He was a visiting scholar at the University of Southern California from October 2018 to April 2019. His research focuses on stochastic, nonsmooth, and nonconvex optimization with applications to machine learning and signal processing.
    \end{IEEEbiography}

      \begin{IEEEbiography}[{\includegraphics[width=1in,height=1.25in,clip,keepaspectratio]{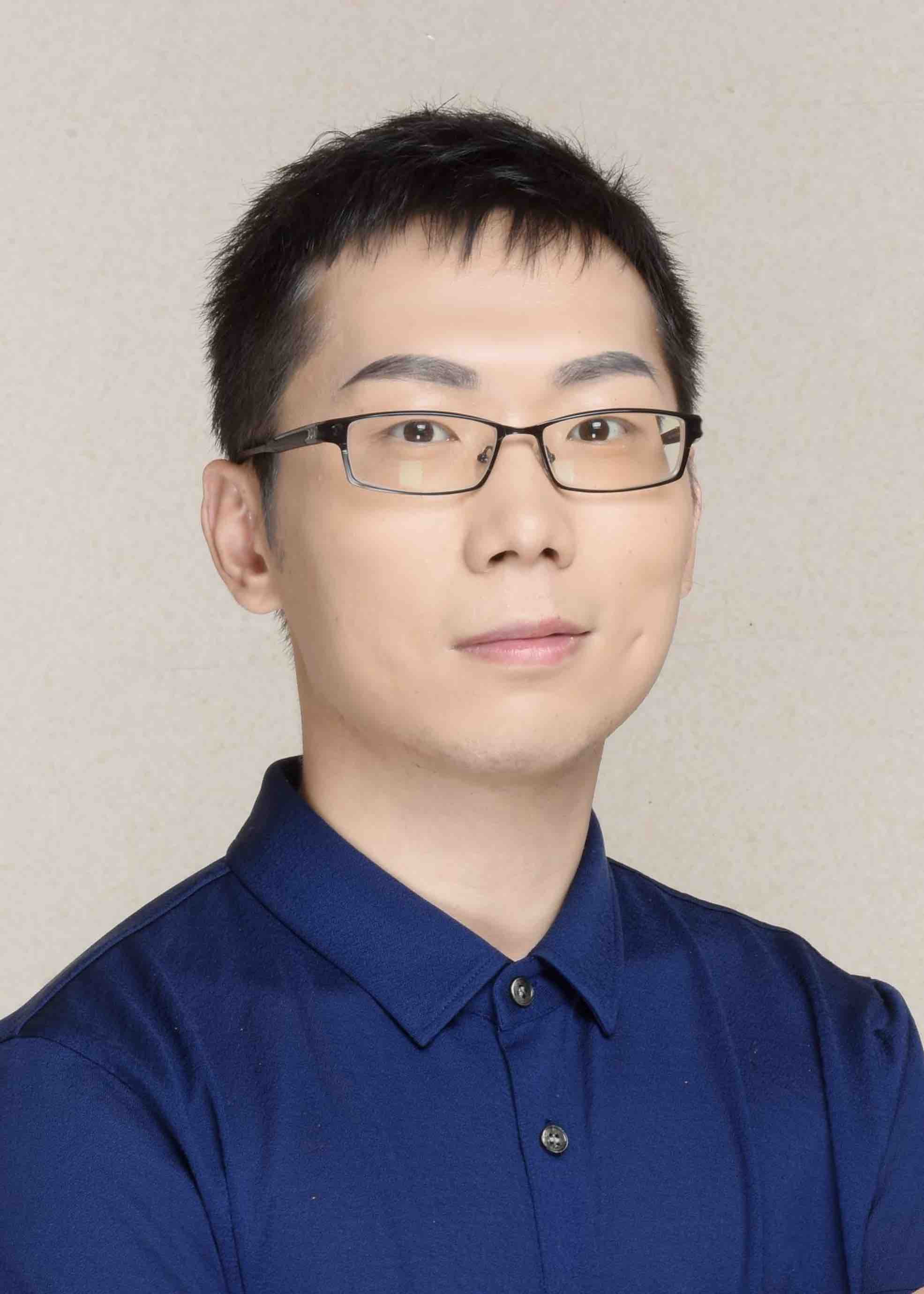}}]
        {Shi Pu}
            is currently an assistant professor in the School of Data Science, The Chinese University of Hong Kong, Shenzhen, China. He is also affiliated with Shenzhen Research Institute of Big Data. He received a B.S. Degree in engineering mechanics from Peking University, in 2012, and a Ph.D. Degree in Systems Engineering from the University of Virginia, in 2016. He was a postdoctoral associate at the University of Florida, Arizona State University and Boston University, respectively from 2016 to 2019. His research
            interests include distributed optimization, network science, machine learning, and game theory.
        \end{IEEEbiography}
    




\end{document}